\documentclass{IEEEojcsys_modified}

\usepackage[colorlinks,urlcolor=blue,linkcolor=blue,citecolor=blue]{hyperref}

\usepackage{color,array}

\usepackage{graphicx}
\IEEEoverridecommandlockouts                 
\usepackage{amsfonts}
\usepackage{subfigure}
\usepackage{graphicx}
\graphicspath{{./IMG/}}
\usepackage{epstopdf}
\usepackage{cite}

\usepackage{comment}
\usepackage{amsmath}
\usepackage{etoolbox}
\usepackage{algorithm}

\AtBeginEnvironment{algorithm2e}{\algoequations}
\AtEndEnvironment{algorithm2e}{\restoreequations}
\newcounter{algosavedequation}
\newcommand{\algoequations}{%
  \setcounter{algosavedequation}{\value{equation}+1}%
  \setcounter{equation}{0}%
  \renewcommand{\theequation}{\arabic{algosavedequation}\alph{equation}}
}
\newcommand{\restoreequations}{%
  \setcounter{equation}{\value{algosavedequation}}%
}

\allowdisplaybreaks[3]

\newtheorem{assumption}{Assumption}
\newtheorem{problem}{Problem}

\newtheorem{remark}{Remark}
\newtheorem{theorem}{Theorem}

\newtheorem{proposition}{Proposition}
\newtheorem{lemma}{Lemma}

\newcommand{\mc}{\mathcal}

\newcommand{\real}{\mathbb{R}}

\newcommand{\naturalpos}{\mathbb{N}_{>0}}
\newcommand{\realpos}{\mathbb{R}_{> 0}}
\newcommand{\realnneg}{\mathbb{R}_{\geq 0}}

\newcommand{\tsp}{\mathsf{T}} 
 
\newcommand{\inv}{{\negat 1}} 
\newcommand{\negat}{\scalebox{0.75}[.9]{\( - \)}}

\newcommand\oprocendsymbol{\hbox{$\square$}}
\newcommand\oprocend{\relax\ifmmode\else\unskip\hfill%
\fi\oprocendsymbol}

\newcommand{\R}{\mathbb{R}} 
\newcommand{\proj}{\Pi_{\mathcal{C}}}

\newcommand{\map}[3]{#1: #2 \rightarrow #3}
\newcommand{\setdef}[2]{\{#1 \; : \; #2\}}

\newcommand{\sbs}[2]{{#1}_{\textup{#2}}}
\newcommand{\sps}[2]{{#1}^{\textup{#2}}}

\newcommand{\W}{\mathcal{W}}

\newcommand*{\QEDB}{\hfill\ensuremath{\square}}
\newcommand*{\QEDBL}{\hfill\ensuremath{\blacksquare}}

\newcommand{\norm}[1]{\Vert #1 \Vert}

\begin{document}


\title{Online Optimization of Dynamical Systems With Deep Learning Perception} 


\author{~Liliaokeawawa Cothren\affilmark{1}}

\author{Gianluca Bianchin\affilmark{1}}

\author{Emiliano Dall'Anese\affilmark{1}}

\affil{Department of Electrical, Computer, and Energy Engineering,
        University of Colorado Boulder, Boulder, CO, USA}



\markboth{ONLINE OPTIMIZATION OF DYNAMICAL SYSTEMS WITH DEEP LEARNING PERCEPTION}{L. Cothren {\itshape ET AL}.}

\begin{abstract}
This paper considers the problem of controlling  a dynamical system when the state cannot be directly measured and the control performance metrics are unknown or partially known. In particular, we focus on the design of data-driven controllers to regulate a dynamical system to the solution of a constrained convex optimization problem where: i) the state must be estimated from nonlinear and possibly high-dimensional data; and, ii) the cost of the optimization problem -- which models control objectives associated with inputs and states of the system -- is not available and must be learned from data. We propose a data-driven feedback controller that is based on adaptations of a projected gradient-flow method; the controller includes neural networks as integral components for the estimation of the unknown functions. Leveraging  stability theory for perturbed systems, we derive sufficient conditions to guarantee exponential input-to-state stability (ISS) of the control loop. In particular, we show that the interconnected system is ISS with respect to the approximation errors of the neural network  and  unknown  disturbances affecting the system. The transient bounds  combine the universal approximation property of deep neural networks
with the ISS characterization. Illustrative numerical results are presented in the context of robotics and  control of epidemics.   
\end{abstract}


\maketitle

\vspace{-.5cm}

\section{Introduction}
\label{sec:introduction}

Control frameworks for modern engineering and societal systems critically rely on 
the use of perceptual information from sensing and estimation mechanisms. 
Extraction of critical information for feedback control increasingly requires the 
processing of high-dimensional sensory data obtained from nonlinear sensory 
systems~\cite{dean2020robust,dean2021certainty,Marchi2022,al2020accuracy}, and the
interpretation of information received from humans interacting with the system 
regarding the end-user perception of safety, comfort, or
(dis)satisfaction~\cite{bajcsy2021analyzing,simonetto2021personalized}. For 
example, control systems in  autonomous driving rely on positioning information 
extracted from camera images~\cite{dean2020robust} and must account for the 
perception of the safety of the vehicle occupants~\cite{Nunen}. In power grids, 
state feedback is derived from nonlinear state estimators or 
pseudo-measurements~\cite{manitsas2008modelling}, and control goals must account 
for comfort and satisfaction objectives of the end-users that are difficult to 
model~\cite{LASSEN2020smiley}. 

Within this broad context, this paper considers the problem of developing feedback controllers for dynamical systems where the acquisition of information on the system state and on the control performance metrics requires a systematic integration of supervised learning methods in the controller  design process. The particular problem that is tackled in this paper pertains to the design of feedback controllers to steer a dynamical system toward the solution of a constrained convex optimization problem, where the cost models objectives are associated with the state and the controllable inputs. The design of feedback controllers inspired by first-order optimization methods has received significant attention during the last decade~\cite{Jokic2009controller,brunner2012feedback,Hirata,lawrence2018optimal,MC-ED-AB:20,zheng2019implicit,hauswirth2020timescale,bianchin2020online,bianchin2021time,hauswirth2021optimization}; see also the recent line of work on using online optimization methods for discrete-time linear time-invariant (LTI) systems~\cite{nonhoff2021data,nonhoff2021online,minasyan2021online,lin2021perturbation}. However, open research questions remain on how it is possible to systematically integrate learning methods in the control loop when information on the system and on the optimization model is not directly available, and on how to analyze the robustness and safety of optimization-based controllers in the presence of learning errors. 

In this work, we investigate the design of feedback controllers, based on an adaptation of the projected gradient flow method~\cite{YSX-JW:00,bianchin2021time}, combined with learning components where: i) estimates of the state of the system are provided by feedforward neural networks~\cite{hornik1989multilayer,barron1994approximation} and residual neural networks~\cite{tabuada2020universal,Tabuada-pmlr-v144-marchi21a}, and ii) the gradient information is acquired via finite-differences based on a deep neural network  approximation of the costs.
When the neural network-based controller is interconnected to the 
dynamical system, we establish conditions that guarantee input-to-state stability \cite{sontag2008input,angeli2003input} by leveraging tools from the theory of perturbed systems~\cite[Ch. 9]{Khalil:1173048}  and singular perturbation theory~\cite[Ch. 11]{Khalil:1173048}. In particular, the ISS bounds show how the transient and asymptotic behaviors of the interconnected system are related to the neural network approximation errors. When the system is subject to unknown disturbances, the ISS bounds account also for the time-variability of the disturbances. 

\emph{Prior works}. Perception-based control of discrete-time linear time-invariant systems is considered in, e.g.,~\cite{dean2020robust,dean2021certainty}, where the authors study the effect 
of state estimation errors on controllers designed via system-level synthesis. Further insights on the tradeoffs between learning accuracy and performance are offered in~\cite{al2020accuracy}. For continuous-time systems, ISS results for dynamical systems with deep neural network approximations of state observers and controllers are provided in~\cite{Marchi2022}. Differently from~\cite{Marchi2022}, we consider the estimation of states and cost functions, and the interconnection of optimization-based controllers with dynamic plants. 
Optimization methods with learning of the cost function are considered in, e.g.,~\cite{simonetto2021personalized,fabiani2021learning, notarnicola2021distributed} (see also references therein); however, these optimization algorithms  are not implemented in closed-loop with a dynamic plant.
Regarding control problems for dynamical systems, existing approaches leveraged gradient-flow controllers~\cite{menta2018stability,bianchin2020online}, proximal-methods in~\cite{MC-ED-AB:20},  prediction-correction 
methods~\cite{zheng2019implicit}, and hybrid accelerated 
methods~\cite{bianchin2020online}. Plants with nonlinear dynamics were considered in~\cite{brunner2012feedback,hauswirth2020timescale}, and switched LTI systems in~\cite{bianchin2021time}. A joint stabilization and regulation problem was considered in~\cite{lawrence2018linear,lawrence2018optimal}. See also the recent survey by~\cite{hauswirth2021optimization}. In all of these works, the states and outputs are assumed to be observable and cost functions are known.

We also acknowledge works where controllers are learned using neural networks; see, for example, \cite{karg2020stability,HY-PS-MA:21,Marchi2022}, and the work on reinforcement learning in~\cite{MJ-JL:20}. Similarly to this literature, we leverage neural networks  to supply state and gradient estimates to a projected gradient-flow controller.  
By analogy with dynamical systems, optimization has been applied to Markov decision processes in e.g.~\cite{PM-JN:01}.

Finally, we note that ISS applied to perturbed gradient flows was investigated 
in~\cite{sontag2021remarks}. In this work, we consider interconnections
between a perturbed, projected gradient-flow and a dynamical  system, and combine 
the theory of perturbed systems~\cite[Ch. 9]{Khalil:1173048}  with singular perturbation~\cite[Ch. 11]{Khalil:1173048}. Our ISS bounds are then 
customized for feedforward neural 
networks~\cite{hornik1989multilayer,barron1994approximation} and residual neural 
networks~\cite{tabuada2020universal,Tabuada-pmlr-v144-marchi21a}. We also acknowledge~\cite{poveda2021data}, where basis expansions are utilized to learn a function, which is subsequently minimized via extremum seeking. 

Finally, the preliminary work~\cite{cothren2021data} used a gradient-flow 
controller in cases where the optimization cost is learned via least-squares 
methods. Here, we extend~\cite{cothren2021data} by accounting for systems with 
nonlinear dynamics, by using neural networks instead than parametric estimation 
techniques, by considering errors in the state estimates, and by combining ISS 
estimates with neural network approximation results.

\emph{Contributions}. The contribution of this work is threefold.
First, we characterize the transient performance of a projected gradient-based 
controller applied to a nonlinear dynamical system while operating with  
errors on the gradient. Our analysis is based on tools from ISS analysis of  nonlinear dynamical systems; more precisely, we leverage Lyapunov-based singular
perturbation reasonings to prove that the proposed control method guarantees 
that the controlled system is ISS with respect to the  variation of exogenous disturbances affecting the system and the error in the gradient. This fact is 
remarkable because unknown exogenous disturbances introduce shifts in the 
equilibrium point of the system to control.
Second, we propose a framework where optimization-based controllers are used in combination with deep neural networks. We tailor our results to two types of deep neural networks that can be used for 
this purpose: deep residual networks and deep feedforward networks. We then combine the universal approximation property of deep neural networks  with the ISS characterization, and we provide an explicit transient bound for 
feedback-based optimizing controllers with neural-network state estimators.
Third, we propose a novel framework where deep neural networks are used to 
estimate the gradients of the cost functions characterizing the control goal  
based on training data. Analogously to the case above, we tailor our results to two cases: deep residual networks and feedforward networks. In this case, we leverage our ISS analysis to 
show how it is possible to design optimization-based controllers 
to accomplish the target control task, and we provide an explicit transient 
bound for these methods.
Finally, we illustrate the benefits of the methods in: (i) an application in robotic control and, (ii) the problem of controlling the outbreak of an
epidemic modeled by using a susceptible-infected-susceptible model.
Overall, our results show for the first time that the universal approximation 
properties of deep neural networks can be harnessed, in combination with the 
robustness properties of feedback-based optimization algorithms, to provide 
guarantees in perception-based control.
 
In conclusion, we highlight that the assumptions and control frameworks outlined in this paper find applications in, for example, power systems~\cite{MC-ED-AB:20,hauswirth2020timescale,menta2018stability,lawrence2018optimal}, traffic flow control in transportation networks~\cite{bianchin2021time}, epidemic control~\cite{bianchin2021planning}, and in neuroscience~\cite{yang2021modelling}; when the dynamical model for the plant does not include exogenous disturbances, our optimization-based controllers can also be utilized in the context of autonomous driving~\cite{dean2020robust,dean2021certainty} and robotics~\cite{terpin2021distributed}.

\emph{Organization}. The remainder of this paper is organized as follows. 
Section~\ref{sec:problemformulation} describes the problem formulation and 
introduces some key preliminaries used in our analysis. 
Section~\ref{sec:error_gradient} presents a main technical result that 
characterizes an error bound for gradient-type controllers with arbitrary 
gradient error.
In Sections~\ref{sec:statePerception} and \ref{sec:costPerception}, we present our
main control algorithms corresponding to the case of state perception and cost 
perception, respectively. Section~\ref{sec:results} illustrates our simulation 
results and Section~\ref{sec:conclusions} concludes the paper.

\section{Preliminaries and Problem Formulation}
\label{sec:problemformulation}

We first outline the notation used throughout the paper and provide relevant definitions.

\emph{Notation}. We denote by \( \mathbb{N}, \naturalpos, \R, \realpos, \text{ and } \realnneg \) the set of natural numbers, the set of positive natural numbers, the set of real numbers, the set of positive real numbers, and the set of non-negative real numbers. For vectors \( x \in \R^n \) and \( u \in \R^m \),
\( \| x \| \) denotes the  Euclidean norm of $x$, \( \| x \|_\infty \) denotes the 
supremum norm, and \( (x,u) \in \R^{n + m}\) denotes their vector concatenation; $x^\top$ denotes transposition, and $x_i$ denotes the $i$-th element of $x$. For a matrix $A \in \R^{n \times m}$, $\|A\|$ is the induced $2$-norm and $\|A\|_\infty$ the supremum norm.

The set $\mathcal{B}_n(r) := \{ z \in \R^{n} : \|z\| < r \}$ is the open ball in 
$\R^{n}$ with radius $r > 0$; $\mathcal{B}_n[r] := \{ z \in \R^{n} : \|z\| \leq r \}$ is the closed ball. Given two sets $\mathcal{X} \subset  \mathbb{R}^n$ and  $\mathcal{Y} \subset  \mathbb{R}^m$, $\mathcal{X} \times \mathcal{Y}$ denotes their Cartesian product; moreover, $\mathcal{X} + \mathcal{B}_n(r)$ is the open set defined as $\mathcal{X} + \mathcal{B}_n(r) = \{x + y: x \in \mathcal{X}, y \in \mathcal{B}_n(r)\}$. 
Given a closed and convex set $\mathcal{C} \subset \R^n$, $ \proj $  denotes the Euclidean projection onto the closed and 
convex set; i.e., $\proj := \arg \min_{x \in \mathcal{C}} \|x - y\|^2$. 
For continuously differentiable
$\map{\phi}{\real^n}{\real}$, \( \nabla \phi(x) \in \real^n\) denotes its gradient.
If the function is not differentiable at a point $x$, $\partial \phi(x)$ denotes 
its subdifferential.

\emph{Partial ordering}. The first orthant partial order on $\R^n$ is denoted as $\preceq$ and it is defined as follows: for any $x, z \in \R^n$, we say that $x \preceq z$ if $x_i \leq z_i$ for $i = 1, \ldots, n$. We say that a function $\phi: \R^n \to \R^n$ is monotone if for any $x, z \in \R^n$ such that $x \preceq z$, we have that $\phi(x) \leq \phi(z)$. Finally, the interval $[x, z]$, for some $x, z \in \R^n$, is defined as $[x, z] = \{w \in \R^n: x \preceq w \preceq z\}$. 

\emph{Set covering}. Let $\mathcal{Q}, \mathcal{Q}_s \subset \R^n$, with $\mathcal{Q}$ compact. We say that $\mathcal{Q}_s$ is an $\varrho$-cover of $\mathcal{Q}$, for some $\varrho > 0$, if for any $x \in \mathcal{Q}$ there exists a $z \in \mathcal{Q}_s$ such that $\|x - z\|_\infty \leq \varrho$.  We say that $\mathcal{Q}_s$ is an $\varrho$-cover of $\mathcal{Q}$ ``with respect to the partial order $\preceq$,'' for some $\varrho > 0$, if for any $x \in \mathcal{Q}$ there exists $w,z \in \mathcal{Q}_s$ such that $x \in [w,z]$ and $\|w - z\|_\infty \leq \varrho$~\cite{Tabuada-pmlr-v144-marchi21a}.

\subsection{Model of the Plant}

We consider systems that can be modeled using continuous-time nonlinear 
dynamics:
\begin{align}
\label{eq:plantModel}
 \dot x & =  f(x, u, w_t), \hspace{.4cm} x(t_0) = x_0 
\end{align}
where 
$f: \mathcal{X} \times \mathcal{U} \times \mathcal{W} \rightarrow \real^{n}$, with
$\mathcal{X} \subseteq \real^{n}$, $\mathcal{U} \subseteq \real^{n_u}$, 
$\mathcal{W} \subseteq \real^{n_w}$ open and connected sets.
In~\eqref{eq:plantModel}, $\map{x}{\realnneg}{\mc X}$ denotes the state, 
$x_0 \in \mathcal{X}$ is the initial condition,  
$\map{u}{\realnneg}{\mc U}$  is the control input, and 
$\map{w_t}{\realnneg}{\mathcal{W}}$ is a time-varying exogenous disturbance 
(the notation $w_t$ emphasizes the dependence on time). 
In the remainder, we restrict our attention to cases where 
$u \in \sbs{\mc U}{c}$ at all times, where $\sbs{\mc U}{c} \subset \mc U$ is 
compact\footnote{The sets $\mathcal{U}_c$ and $\mathcal{W}_c$ are assumed compact to reflect hardware and operational constraints in applications such as autonomous driving~\cite{dean2020robust}, power  
systems~\cite{hauswirth2020timescale,MC-ED-AB:20}, neuroscience~\cite{yang2021modelling}, control of 
epidemics~\cite{GB-ED-JP-AB:21-scirep}.}; moreover, we assume that the vector 
field $f(x,u,w)$ is continuously-differentiable and Lipschitz-continuous, with 
constants $L_x$, $L_u$, $L_w$, respectively, in its variables. 
We make the following assumptions on~\eqref{eq:plantModel}.

\vspace{-.3cm}

\begin{assumption}[\bf \textit{Steady-state map}]
\label{as:steadyStateMap}
There exists a (unique) continuously differentiable function
$h: \mathcal{U} \times \mathcal{W} \to \mathcal{X}$ 
such that, for any fixed $\bar u \in \mathcal{U}, \bar w \in \mathcal{W}$, 
$f(h(\bar u,\bar w), \bar u, \bar w) = 0$.
Moreover, $h(u,w)$ admits the decomposition $h(u,w) = h_u(u) + h_w(w)$, where 
$h_u$ and $h_w$ are Lipschitz continuous with constants $\ell_{h_u}$ and 
$\ell_{h_w}$, respectively. 
\QEDB \end{assumption}
\vspace{-.2cm}

Assumption~\ref{as:steadyStateMap} guarantees that, with constant inputs
$\bar u, \bar w$, system \eqref{eq:plantModel} 
admits a unique equilibrium point $\bar x := h(\bar u,\bar w)$.
Notice that existence of $h(u,w)$ is always guaranteed in cases where, in 
addition, $\nabla_x f(x, \bar u, \bar w)$ is invertible for any $\bar u, \bar w$.
Indeed, in these cases, the implicit function 
theorem~\cite{hauswirth2020timescale} guarantees that  $h(u,w)$ exists and is 
differentiable, since $f(x,u,w)$ is continuously differentiable.

In this work, we interpret $w_t$ as an unknown  exogenous input modeling disturbances affecting the system. We make the following assumption on $w_t$.

\begin{figure}[t]
\centering \includegraphics[width=.8\columnwidth]{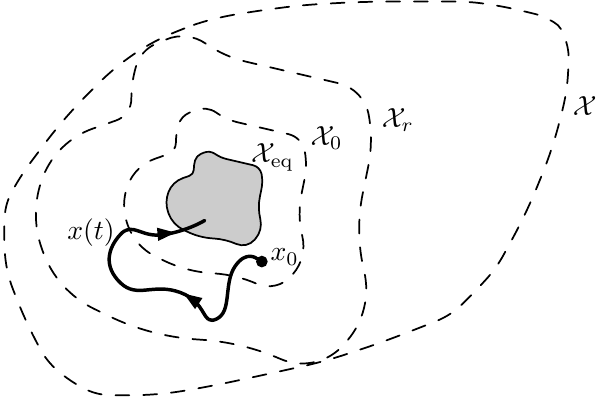}
\caption{Illustration of notation used for sets and
Assumption~\ref{as:stabilityPlant}. Continuous lines denote compact sets, 
dashed lines denote open sets. Assumption~\ref{as:stabilityPlant} guarantees
that trajectories that start from the set of initial conditions $\mc X_{0}$ do not leave $\sbs{X}{r}$.}
\label{fig:sets_illustration}
\end{figure}

\vspace{-.3cm}

\begin{assumption}[\bf\textit{Properties of exogenous inputs}]
\label{as:disturbance}
For all $t \in \realnneg$, $w_t \in \sbs{\mathcal{W}}{c}$, where 
$\sbs{\mathcal{W}}{c} \subset  \mathcal{W}$ is compact. Moreover, 
\( t \mapsto w_t\) is locally absolutely continuous on $\sbs{\mathcal{W}}{c}$.~ \QEDB
\end{assumption}

\vspace{-.3cm}

Assumption~\ref{as:disturbance} imposes basic continuity and compactness 
requirements on the exogenous disturbances affecting~\eqref{eq:plantModel}.
Following Assumption~\ref{as:disturbance}, in the remainder of this paper we 
denote by $\sbs{\mathcal{X}}{eq} := h(\sbs{\mc U}{c} \times \sbs{\mc W}{c})$ the 
set of admissible equilibrium points of the system~\eqref{eq:plantModel}. 
We note that in Assumption~\ref{as:steadyStateMap} that we consider a decomposition $h(u,w) = h_u(u) + h_w(w)$ so that the Jacobian of $h(u,w)$ with respect to $u$ does not depend on the unknown disturbance $w_t$; this property will be leveraged in the implementation of our gradient-based controller. Notably, this assumption is satisfied in, e.g., power systems~\cite{MC-ED-AB:20,hauswirth2020timescale,menta2018stability,lawrence2018optimal}, transportation networks~\cite{bianchin2021time}, and in neuroscience~\cite{yang2021modelling}. Our model clearly subsumes the case where no disturbance $w$ is present, as in the models for, e.g., autonomous driving~\cite{dean2020robust,dean2021certainty} and robotics~\cite{terpin2021distributed}. We also emphasize that the  dynamics~\eqref{eq:plantModel} can model both the dynamics of the physical system and of the stabilizing controllers; see, for example,~\cite{lawrence2018optimal}, our previous work on LTI systems in~\cite{galarza2022online}, and the recent survey~\cite{hauswirth2021optimization}.

\begin{remark}[\bf\textit{Compactness of the equilibrium set}]
Notice that the equilibrium set $\sbs{\mathcal{X}}{eq}$ is compact. This follows 
by noting that $\sbs{\mathcal{U}}{c} \times \sbs{\mathcal{W}}{c}$ is compact, 
$h(u,w)$ is continuously differentiable, and by application 
of~\cite[Theorem 4.14]{rudin1976principles}. Moreover, notice that  
$\| \nabla_u h(u,\bar w)\| \leq \ell_{h_u}$ for all $ u\in\sbs{\mathcal{U}}{c}$, 
which follows from compactness of $\sbs{\mathcal{U}}{c}$, 
see~\cite[Ch. 4]{rudin1976principles}. 
\QEDB\end{remark}
\vspace{-.3cm}

Before proceeding, we let $r$ denote the largest positive constant such that
$\mc X_r :=\sbs{\mc X}{eq} + \mc B_n(r)$ satisfies  $\mc X_r \subseteq  \mc X$ 
(see Figure~\ref{fig:sets_illustration} for an illustration).
For instance, if $\mc X = \setdef{x \in \real^n}{\norm{x}<\varrho}$ and 
$\sps{\mc X}{eq} = \{0\},$ then $r = \varrho$.

\vspace{-.2cm}
\begin{assumption}[\bf\textit{Exponential stability}]
\label{as:stabilityPlant}
There exist $a, k >0$ such that for any fixed 
$\bar u \in \sbs{\mathcal{U}}{c},\bar w \in \sbs{\mathcal{W}}{c}$, the bound
\begin{align}
\label{eq:expstability}
\|x(t) - h(\bar u, \bar w)\| \leq 
k \|x_0 - h(\bar u, \bar w)\| e^{-a (t - t_0)} ,
\end{align}
holds for all $t \geq t_0$ and for every initial condition 
$x_0 \in \mc X_{0} := \sbs{\mathcal{X}}{eq} + \mathcal{B}_n(r_0)$, $r_0<(r-\text{diam}(\sbs{\mathcal{X}}{eq} ))/k$, 
where $x(t)$ is the solution of \eqref{eq:plantModel} with  $u(t)=\bar u$ and 
$w(t)=\bar w$.
\QEDB\end{assumption}
\vspace{-.2cm}

Assumption~\ref{as:stabilityPlant} 
guarantees that $\bar x=h(\bar u, \bar w)$ is 
exponentially stable, uniformly in time. This, in turn, implies the existence of 
a Lyapunov  function as formalized in the following result, which is a direct application of~\cite[Thm.~4.14]{Khalil:1173048}.



\vspace{-.2cm}

\begin{lemma}[\bf \textit{Existence of a Lyapunov function for~\eqref{eq:plantModel}}]
\label{lem:converse_plant}
Let Assumptions~\ref{as:steadyStateMap}-\ref{as:stabilityPlant} hold and let $\mc X_0$ be the set of initial conditions as in Assumption~\ref{as:stabilityPlant}.
Then, there exists a function  
$W: \mathcal{X}_0 \times \mathcal{U} \times \mathcal{W} \to \R$ that satisfies 
the inequalities: 
\begin{align}
\label{eq:converse_plant}
&\quad\quad\quad  d_1 \|x - h(u,w)\|^2  \leq W(x,u,w) \leq d_2 \|x - h(u,w)\|^2, \nonumber\\
& \frac{\partial W}{\partial x}f(x,u,w)  \leq -d_3\|x - h(u,w)\|^2, \nonumber\\ 
&\left\| \frac{\partial W}{\partial x}\right\| \leq d_4 \|x - h(u,w)\|,\quad
\left\| \frac{\partial W}{\partial u}\right\| \leq d_5 \|x - h(u,w)\|,\nonumber\\
& \left\| \frac{\partial W}{\partial w}\right\| \leq d_6 \|x - h(u,w)\|,
\end{align}
for some positive constants $d_1\leq d_2, d_3, d_4,d_5,d_6$.  
\QEDB
\vspace{-.3cm}
\end{lemma}

\begin{proof}
We begin by noting that, under our assumptions, the vector field $f(x,u,w)$ is 
Lipschitz in 
$\mathcal{X}_r \times \sbs{\mathcal{U}}{c} \times \sbs{\mathcal{W}}{c}$, and thus
its  Jacobian $\frac{\partial f}{\partial x}$ is bounded on $\mathcal{X}_r$, 
uniformly with respect to $u$ and $w$. The proof thus follows by 
iterating the steps in~\cite[Thm.~4.14]{Khalil:1173048} for fixed 
$u \in \sbs{\mathcal{U}}{c}$ and $w \in \sbs{\mathcal{W}}{c}$, by noting that
Assumption~\ref{as:stabilityPlant} implies that solutions that start in 
$\mathcal{X}_0$ do not leave $\mathcal{X}_r$, and thus~\eqref{eq:expstability} 
holds. Then, sensitivity with respect to $u,w$ follows 
from~\cite[Lemma~9.8]{Khalil:1173048} and~\cite{Corringendum}.~
\end{proof}

In the following, we state the main optimization problem associated with~\eqref{eq:plantModel} and formalize the problem statements.  

\subsection{Target Control Problem}

In this work, we focus on the problem of controlling, at every time $t$, the 
system~\eqref{eq:plantModel} to a solution of the following time-dependent 
optimization problem:
\begin{subequations}
\label{opt:objectiveproblem}
\begin{align}
\label{opt:objectiveproblem-a}
(u_t^*, x_t^*) \in  \arg
\underset{\bar u, \bar x}{\min}  ~~ & 
\phi (\bar u) + \psi (\bar x)\\
\label{opt:objectiveproblem-b}
\text{s.t.} ~~~~ & \bar x = h(\bar u , w_t), \quad \bar u \in \mc C,
\end{align}
\end{subequations}
where $\map{\phi}{\mc U}{\real}$ and $\map{\psi}{\mc X}{\real}$ 
describe costs associated with the system's inputs and states, respectively,
and $ \mathcal{C} \subset \mc U_c$ is a closed and convex set representing 
constraints on the input at optimality.

\begin{remark}[\bf\textit{Interpretation of the control objective}]
The optimization problem \eqref{opt:objectiveproblem} formalizes an 
\emph{optimal equilibrium selection problem}, where the objective is to select
an optimal input-state pair $(u^*_t,x^*_t)$ that, at equilibrium, minimizes 
the cost specified by $\phi(\cdot)$ and $\psi(\cdot)$. 
It is worth noting that, differently from stabilization problems, where the 
objective is to guarantee that the trajectories of~\eqref{eq:plantModel} 
converge to 
\textit{some} equilibrium point, the control objective formalized by is to \textit{select, among all equilibrium points of~\eqref{eq:plantModel}, an equilibrium point that is optimal} as described by the function $\phi(u) + \psi(x)$.
In this sense, \eqref{opt:objectiveproblem} can be interpreted as a high-level control objective that 
can be nested with a stabilizing controller (where the latter is used to guarantee the satisfaction of Assumption~\ref{as:stabilityPlant}).
\vspace{-.3cm}
\QEDB\end{remark}

Two important observations are in order. First, the 
constraint~\eqref{opt:objectiveproblem-b} is parametrized by the
disturbance $w_t$, and thus the solutions 
of~\eqref{opt:objectiveproblem} are parametrized by $w_t$ (or, equivalently, 
by time). In this sense, the pairs $(u_t^*, x_t^*)$ are time-dependent and 
characterize optimal trajectories~\cite{popkov2005gradient}.
Secondly, by recalling that $w_t$ is assumed to be unknown and unmeasurable,
solutions of~\eqref{opt:objectiveproblem} cannot be computed~explicitly.

By recalling that $h(u, w)$ is unique for any fixed $u,w$, 
problem~\eqref{opt:objectiveproblem} can be rewritten as an unconstrained problem:
\begin{align}
\label{opt:objectiveproblem-2}
u_t^* \in  
\arg \underset{\bar u \in \mc C}{\min}  ~~ & 
\phi (\bar u) + \psi (h(\bar u , w_t)) \, .
\end{align}

We make the following assumptions on the costs of~\eqref{opt:objectiveproblem-2}.

\vspace{-.2cm}

\begin{assumption}[\bf \textit{Smoothness and strong convexity}]
\label{as:LipsConvex} The following conditions hold:
\begin{enumerate}
\item[(a)] The function $u \mapsto \phi(u)$ is continuously-differentiable and 
$\ell_u$-smooth, $\ell_u \geq 0$.
The function $x \mapsto \psi(x)$ is continuously-differentiable and 
$\ell_x$-smooth, $\ell_x \geq 0$.

\item[(b)] For any $\bar w \in \sbs{\mc W}{c}$ fixed, the composite cost
$u \mapsto \phi(u) + \psi(h(u,\bar w))$ is  $\mu_u$-strongly  convex, $\mu_u > 0$.
\QEDB
\end{enumerate}
\end{assumption}
\vspace{-.2cm}

It follows from Assumption~\ref{as:LipsConvex}(a) that the composite cost
$u \mapsto \phi(u) + \psi(h(u,w_t))$ is $\ell$-smooth with 
$\ell:= \ell_u+\ell_{h_u}^2 \ell_x$; it follows from Assumption~\ref{as:LipsConvex}(b) 
that the optimizer $(u^*_t,x^*_t)$ of~\eqref{opt:objectiveproblem-2} is 
unique for any $w_t \in \W$.

\vspace{-.3cm}

\begin{assumption}[\bf \textit{Regularity of optimal trajectory map}]
\label{as:mapu}
There exists a continuous function 
$J: \sbs{\mathcal{W}}{c} \to \sbs{\mathcal{U}}{c}$ such that $u^*_t = J(w_t)$. 
Moreover, there exists $\ell_J < \infty$ such that 
$\|\partial J(w_t)\| \leq \ell_J$ for all $w_t \in\sbs{\mathcal{W}}{c}$.~
\QEDB
\end{assumption}
\vspace{-.3cm}

Assumption~\ref{as:mapu} imposes regularity assumptions on the function that maps 
$w_t$ (which parametrizes the problem~\eqref{opt:objectiveproblem-2}) into the 
optimal solution $u^*_t$~\cite[Ch.~2]{dontchev2009implicit}; conditions can be obtained from  standard
arguments in parametric convex programming.

\subsection{Optimal Regulation with Perception In-the-loop}

Feedback-based optimizing controllers for 
\eqref{eq:plantModel}-\eqref{opt:objectiveproblem} were studied 
in~\cite{bianchin2021time} when~\eqref{eq:plantModel} has linear dynamics 
and in~\cite{hauswirth2020timescale} when~\eqref{opt:objectiveproblem} is 
unconstrained and $w_t$ is constant. The authors consider low-gain gradient-type 
controllers of the form:
\begin{align}
\label{eq:ideal_controller}
\dot u  = \proj \left\{u-\eta \left(\nabla \phi(u) 
+ H(u)^\top  \nabla \psi (x) \right) \right\}-u,
\end{align}
where $H(u)$ denotes the Jacobian of $h_u(u)$ and  $\eta > 0$ is a tunable 
controller parameter.
The controller~\eqref{eq:ideal_controller} is of the form of a projected 
gradient-flow algorithm, often adopted to solve problems of the 
form~\eqref{opt:objectiveproblem}, yet modified by replacing the true gradient
$\nabla \psi (h(u,w_t))$ with the gradient $\nabla \psi (x)$ evaluated at the instantaneous 
system state, thus making the iteration~\eqref{eq:ideal_controller} independent of the 
unknown disturbance $w_t$. 

Implementations of the controller~\eqref{eq:ideal_controller} critically rely on
the exact knowledge of the system state $x$ as well as of the gradients
$\nabla \phi(u)$ and $\nabla \psi (x)$.
In this work, we consider two scenarios. 
In the first, the controller is used with an estimate $\hat x$ of $x$ provided by a deep 
neural network. More precisely, we focus on cases where $x$ is not directly measurable, 
instead, we have access only to nonlinear and possibly high-dimensional observations of the 
state  $\xi = q(x)$, where $q: \mathcal{X} \rightarrow \real^{n_\xi}$ is an \emph{unknown} map. 
In the second case, the controller is used with estimates of the gradients 
$\nabla \phi(u)$,  $\nabla \psi (x)$, obtained by using a deep neural network.
More precisely, we consider cases where the analytic expressions of the gradients are 
unknown, instead, we have access only to functional evaluations $\{u_i, \phi(u_i)\}$,  
$\{x_i, \psi(x_i)\}$ of the cost functions. 
We formalize these two cases next.

\vspace{-.3cm}

\begin{problem}[\textbf{\textit{Optimization with state perception}}]
\label{pro:perception}
Design a feedback controller to regulate inputs and states 
of~\eqref{eq:plantModel} to the time-varying solution 
of~\eqref{opt:objectiveproblem} when $x$ is unmeasurable and, instead, we have 
access only to state estimates $\hat{x} = \hat p(\xi)$ produced by a deep neural 
network $\hat p(\cdot)$ trained as a state observer.
\QEDB
\vspace{-.3cm}
\end{problem}

\begin{problem}[\textbf{\textit{Optimization with cost perception}}]
\label{pro:cost}
Design a feedback controller to regulate inputs and states 
of~\eqref{eq:plantModel} to the time-varying solution 
of~\eqref{opt:objectiveproblem} when $\nabla \phi(u)$, $\nabla \psi (x)$ are 
unknown and, instead, we have access only to estimates $\hat \phi(u)$, 
$\hat \psi (x)$ of $\phi(u)$, $\psi (x)$ produced by a deep neural~network trained
as a function estimator.
\QEDB\end{problem}

\vspace{-0.2cm} 

We conclude by discussing in the following remarks the relevance of 
Problems~\ref{pro:perception}-\ref{pro:cost} in the applications.

\vspace{-0.2cm} 

\begin{remark}[\bf\textit{Motivating applications for Problem~\ref{pro:perception}}]
In applications in autonomous driving, vehicles states are often reconstructed from 
perception-based maps $\xi = q(x)$ where $q$ describes images generated by cameras.
In a power systems context, $\xi=q(x)$ describes the highly-nonlinear power flow 
equations describing the relationships between net powers and voltages 
at the buses (described by $\xi$) and generators' phase angles and frequencies 
(described by $x$).
Finally, we note that a related observer design problem was considered in~\cite{Marchi2022}.
\vspace{-.3cm}
\QEDB\end{remark}

\begin{remark}[\bf\textit{Motivating applications for Problem~\ref{pro:cost}}]
When systems interact with humans, $\phi(u)$ is often used to model end-users' 
perception regarding safety, comfort, or (dis)satisfaction of the adopted 
control policy~\cite{simonetto2021personalized, notarnicola2021distributed,fabiani2021learning,luo2020socially,cothren2021data}. 
Due to the complexity of modeling humans, $\phi(u)$ is often unknown and learned 
from available historical data.
In robotic trajectory tracking problems, $\psi(x) = \|x - x^{r}\|^2$ where 
$x^{r} \in \mathbb{R}^p$ models an unknown target to be tracked. In these cases, we have 
access only to measurements of the relative distance $\|x - x^{r}\|^2$ between the robot the 
target.
Additional examples include cases where $\psi(x)$ represents a barrier function 
associated with unknown sets~\cite{robey2020learning,taylor2020learning}. 
\QEDB \end{remark}

\section{General Analysis of Gradient-Flow Controllers with Gradient Error}
\label{sec:error_gradient}

In this section, we take a holistic approach to address 
Problems~\ref{pro:perception}-\ref{pro:cost} and we provide a general result characterizing
gradient-type controllers of the form~\eqref{eq:ideal_controller} that operate with general 
errors. More precisely, in this section we study the following plant-controller 
interconnection:
\begin{subequations}
\label{eq:closedloop_error}
\begin{align}
\dot x & =  f(x, u, w_t),  \\
\dot u & = \proj \left\{u-\eta \left( F(x,u) + e(x,u) \right) \right\}-u, 
\label{eq:closedloop_error_b}
\end{align}
\end{subequations}
with $x(t_0)=x_0$ and $u(t_0)=u_0$, where 
$F(x,u) := \nabla \phi(u) + H(u)^\top  \nabla \psi (x)  $ is the nominal 
gradient as in~\eqref{eq:ideal_controller}, and 
$\map{e}{\mc X \times \mc U}{\R^{n_u}}$ models any state- or 
input-dependent~error.

It is worth noting three important features of the controller~\eqref{eq:closedloop_error_b}.
First, \eqref{eq:closedloop_error_b} can be implemented without knowledge of $w_t$
(similarly to~\eqref{eq:ideal_controller}, the true gradient $\nabla \psi (h(u,w_t))$ 
is replaced by evaluations of the gradient at the instantaneous state $\nabla \psi (x)$). 
Second, since the vector field in~ \eqref{eq:closedloop_error_b} is Lipschitz-continuous,
for any $(x_0,u_0)$ the initial value problem~\eqref{eq:closedloop_error} admits a unique 
solution that is continuously 
differentiable~\cite[Lemma 3.2]{bianchin2021time},\cite{YSX-JW:00}.
Third, the set $\mathcal{C}$ is attractive and forward-invariant for the 
dynamics~\eqref{eq:closedloop_error_b}, namely, if $u(t_0) \not \in \mc U$, then $u(t)$  
approaches $\mc C$ exponentially, and if $u(t_0) \in \mc C$, then $u(t) \in \mc U$ for all 
$t\geq t_0$~\cite[Lemma 3.3]{bianchin2021time}.

To state our results, we let $z:=(x - x_t^*, u - u_t^*)$ be the tracking error 
between the state of~\eqref{eq:closedloop_error} and  the optimizer 
of~\eqref{opt:objectiveproblem}. Moreover, for fixed $s \in (0,1)$, define:
\begin{align}
\label{eq:uglyConstants}
c_0 &:= \min \{s \mu \eta, s \frac{d_3}{d_2} \}, &
c_1 &:= \frac{1}{\eta}\min\left\{\frac{(1-\theta)}{2}, \theta d_1 \right\}, \nonumber \\
c_2 &:= \frac{1}{\eta} \max\left\{ \frac{(1-\theta)}{2}, \theta d_2  \right\}, &
c_3 &:= \sqrt{2} c_1^{-1}, \nonumber \\    
c_4 & := \sqrt{2 \eta} \max\{1,d_1^{-\frac{1}{2}}\}, &
c_5 &:= \frac{\sqrt{2}}{\sqrt{\eta}}\ell_J + \frac{d_4 \ell_{h_w}}{\sqrt{\eta}\sqrt{d_1}},
\end{align}
and $\theta = (1+d_4+\frac{d5}{\ell \ell_{h_u}})^\inv$, where we recall that 
$(d_1,d_2,d_3,d_4,d_5)$ are as in Lemma~\ref{lem:converse_plant}.

\vspace{-.2cm}
\begin{theorem}[\bf \textit{Transient bound for gradient flows with error}]
\label{thm:General_stability}
\vspace{-.8cm}

\noindent
Consider the closed-loop system~\eqref{eq:closedloop_error} and let Assumptions \ref{as:steadyStateMap}-\ref{as:mapu} be satisfied. Suppose that, for any $x \in \mathcal{X}_0$ and $u \in \mathcal{U}$, the gradient error satisfies the condition 
\begin{align}\label{eq:errorcondition}
\|e(x,u)\| \leq \gamma \|z\| + \delta,
\end{align}
for some $\delta > 0$ and $\gamma \in [0, c_0/c_3)$.  If $\eta \in (0,\eta^*)$,
with 
\begin{align}\label{eq:ControllerGain}
\eta^* := \min \left\{ \frac{2 \mu}{\ell^2}, \frac{(1 - s)^2 d_3 \mu}{
\ell_{h_u}(d_4 \ell_{h_u} +d_5)((1-s)\mu \ell_y + 2 \ell^2)} \right\} ,
\end{align}
then, the tracking error satisfies
\begin{align}\label{eq:genErrorBound}
    \|z(t)\| \leq & \kappa_1 e^{-\frac{\alpha}{2}(t - t_0)} \|z(t_0)\| 
    + \kappa_2\textrm{ess sup}_{t_0 \leq \tau \leq t} \|\dot w_{\tau}\| 
    + \kappa_3  \delta,
\end{align}
for all $t \geq t_0$, where $\alpha = c_0 - \gamma c_3$ and
\begin{align*}
\kappa_1 &= \sqrt{\frac{c_2}{c_1}}, &
\kappa_2 &= \frac{c_5}{c_0 \sqrt{c_1}}, & 
\kappa_3 &= \frac{c_4}{c_0 \sqrt{c_1}},    
\end{align*}
for any  $x(t_0) \in \mathcal{D}_0  := \sbs{\mathcal{X}}{eq} + \mathcal{B}_n(r')$ where $r'$ is such that $0 < r' < r_o/(\sqrt{2} \kappa_1) - \sqrt{2} \kappa_3 \delta -  \sqrt{2} \kappa_2 \textrm{ess sup}_{\tau \geq t_0} \|\dot w_{\tau}\| - \sqrt{2} \kappa_1 \textrm{diam}(\mathcal{U})$, and for any $u(t_0) \in \mathcal{U}$.
\QEDB \end{theorem}
\vspace{-.2cm}
The proof of this claim is postponed to the Appendix.

Theorem~\ref{thm:General_stability} asserts that if the worst-case estimation error $e(x,u)$ 
is bounded by a term $\gamma \|z\|$ that vanishes at the optimizer and a by a nonvanishing 
but constant term $\delta$, then a sufficiently-small choice of the gain $\eta$ 
guarantees exponential convergence of the tracking error to a neighborhood of zero. 
More precisely, the tracking error $z$ is ultimately bounded by two terms: the first 
$\kappa_2\textrm{ess sup}_{t_0 \leq \tau \leq t} \|\dot w_{\tau}\|$ accounts for the effects 
of  the time-variability of $w_t$ on the optimizer $(u_t^*,x_t^*)$, and the second 
$\kappa_3 \delta$ accounts for the effects of a nonvanishing error in the utilized gradient 
function.
It follows that the bound \eqref{eq:genErrorBound} guarantees 
input-to state stability (ISS) of \eqref{eq:closedloop_error} (in the sense
of~\cite{sontag1997output,angeli2003input,sontag2021remarks}) with respect to 
$\norm{\dot w_t}$ and~$\delta$.

\section{Optimization with Neural Network State~Perception}
\label{sec:statePerception}
In this section, we propose an algorithm to address 
Problem~\ref{pro:perception}, and we tailor the conclusions drawn in  
Theorem~\ref{thm:General_stability} to characterize the performance of the 
proposed algorithm.

\subsection{Algorithm Description}
\label{sec:unknownobservation}

To produce estimates of the system state $\hat x = p(\xi)$, we assume that a set
of training points $\{(\xi^{(i)}, x^{(i)})\}_{i = 1}^N$ is utilized to train a 
neural network via empirical risk minimization. 
More precisely, in the remainder, we will study two types of neural networks that
can be used for this purpose: (i) feedforward neural networks, and (ii) residual
neural networks\footnote{We refer the reader to the representative papers~\cite{hornik1989multilayer,Tabuada-pmlr-v144-marchi21a} for an overview of feedforward and residual networks.
Briefly, a neural network consists of inputs, various hidden layers, activation functions, and output layers, and can be trained for, e.g., functional estimation and classification.  When the  layers are sequential and the architecture is described by a directed acyclic graph, the underlying network is called \textit{feedforward}; when some of these layers are bypassed, then the underlying network is called \textit{residual}.}.
We thus propose to train a neural network to produce a map 
$\hat{x} = \hat p(\xi)$ that yields estimates of the system state given 
nonlinear and high-dimensional observations $\xi$. Accordingly, we modify the 
controller~\eqref{eq:ideal_controller} to operate with estimates of the system 
state $\hat x$ produced by the neural network. The proposed framework is described  in Algorithm~\ref{alg:perception_cost} and illustrated in Figure~\ref{fig:system_1}. 
\begin{algorithm}
\caption{Optimization with NN State Perception}

\label{alg:perception_cost}
\# \textbf{Training}

\quad Given: training set $\{(x^{(i)}, \xi^{(i)})\}_{i = 1}^N$

\quad Obtain: $\hat p \gets \operatorname{NN-learning}(\{(x^{(i)}, \xi^{(i))}\}_{i = 1}^N)$ 

\vspace{.2cm}

\# \textbf{Gradient-based Feedback Control}

\quad Given: set $\mathcal{U}$, funct.s $\nabla \phi, \nabla \psi, H(u)$, 
neural net $\hat p$, gain~$\eta$

\quad Initial conditions: $x(t_0) \in \mathcal{X}_0$, $u(t_0) \in \mathcal{U}$

\quad For $t \geq t_0$:
\begin{subequations}
\label{eq:closedloop_perception}
\begin{align}
 \dot x & =  f(x, u, w_t)  \\
      \xi & = q(x) \\
      \dot u & = \proj \hspace{-.1cm} \left\{u-\eta \left(\nabla \phi(u) + H(u)^\top  \nabla \psi (\hat{p}(\xi)) \right) \right\}-u \hspace{-.2cm} \label{eq:closedloop_perception_c}
\end{align}
\end{subequations}

\end{algorithm}

\begin{figure}[h!]
\centering \includegraphics[width=.9\columnwidth]{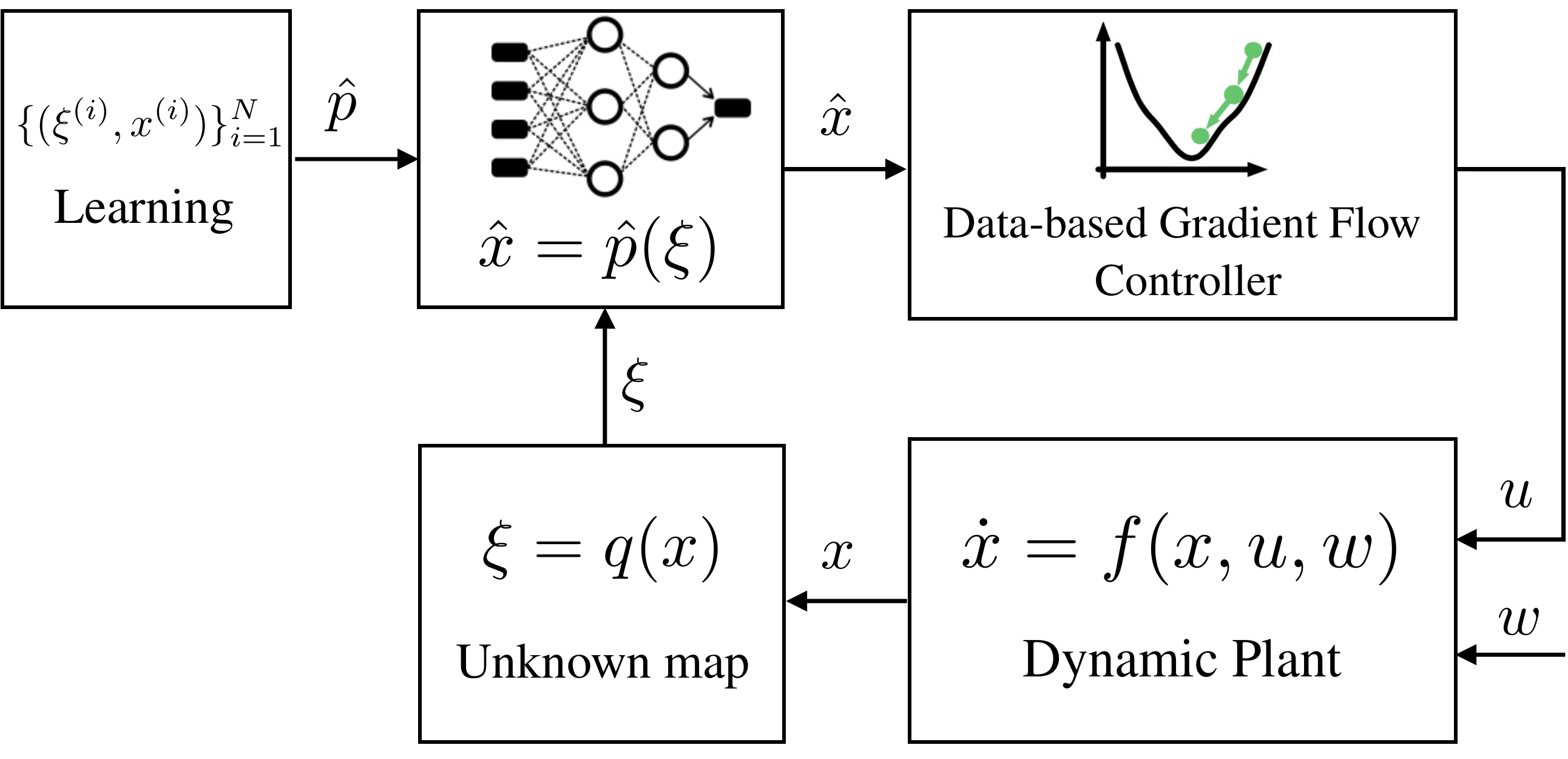}
\caption{Schematics of the control framework proposed in 
Algorithm~\ref{alg:perception_cost}. In this setting, the state 
$x$ of the system is not measurable directly, instead, only 
high-dimensional estimations (e.g., camera images) $\xi = q(x)$ 
are available.  A deep neural network is then used to compute 
state estimates $\hat{x} = \hat p(\xi)$, which are used to feed a gradient-based feedback controller).}
\label{fig:system_1}
\end{figure}

In the training phase of Algorithm~\ref{alg:perception_cost}, the map 
$\operatorname{NN-learning}(\cdot)$ denotes a generic training procedure for 
the neural network via empirical risk minimization. 
The output of the training phase is the neural network mapping $\hat p(\cdot)$. 
In the feedback control phase, the map $\hat p(\cdot)$ is then used to produce 
estimates of the state of the dynamical system $\hat x = \hat p(\xi)$ in order 
to evaluate the gradient functions.
Notice that, relative to the nominal controller~\eqref{eq:ideal_controller}, 
\eqref{eq:closedloop_perception_c} leverages a gradient that is evaluated at an 
approximate point $\hat x$, and thus fits the more general 
model~\eqref{eq:closedloop_error_b}.

\subsection{Analysis of Algorithm~\ref{alg:perception_cost}}

In what follows, we analyze the tracking properties of 
Algorithm~\ref{alg:perception_cost}.
To this end, we introduce the following.

\begin{assumption}[\bf\textit{Generative and Perception Maps}]
\label{as:perceptionmap}
The generative map $x \mapsto q(x)=\xi$ is such that, for any compact set $\mathcal{X}' \subseteq \mathcal{X}_r$,  the image $q(\mathcal{X}')$ is compact.
Moreover, there exists a continuous $p: \real^{n_\xi} \to \real^n$ such that 
$p(\xi) = x$ for any $x \in \mathcal{X}_r$, where $\xi = q(x)$.
\QEDB \end{assumption}

\vspace{-.3cm}

\begin{remark}[\bf\textit{Relationship with System Observability}]
We note that a standard approach in the literature for the state observer design problem is to leverage the concept of $\ell$-observability, where the state is estimated based on $\ell +1$ samples of $\xi$ and $\ell$ samples of the inputs $u, w$~\cite{Marchi2022}. However, we note that in our setup we do not have access to measurements of the exogenous input $w_t$; therefore, we rely on an approach similar to~\cite{dean2020robust, dean2021certainty} where $x$ is estimated from the observation $\xi$.    
\QEDB\end{remark}
\vspace{-.2cm}

To guarantee that network training is well-posed, we assume that the 
$N$ training points $\{x^{(i)}\}_{i = 1}^N$ for the state are drawn from a 
compact set $\mathcal{X}_{\textrm{train}} := \sbs{\mathcal{X}}{eq} + \mathcal{B}_n[r_\textrm{train}]$, where $r_\textrm{train}$ is such that  
$r_0 \leq r_\textrm{train} < r$. 
Moreover, we let
$\mathcal{Q}_\textrm{train} := q(\mathcal{X}_\textrm{train})$ 
denote the perception set associated with the training  set 
$\mathcal{X}_\textrm{train}$, and we denote by 
$\mathcal{Q}_{\textrm{train},s} := \{\xi^{(i)} = q(x^{(i)}), i = 1, \ldots, N\} \subset \mathcal{Q}_{\textrm{train}}$ the set of available perception samples.
Notice that the set $\mathcal{Q}_\textrm{train}$ is compact by
Assumption~\ref{as:perceptionmap}. Compactess of $\mathcal{Q}_\textrm{train}$ 
will allow us to build on the results of~\cite{hornik1989multilayer} 
and~\cite{Marchi2022,tabuada2020universal} to bound the perception error 
$\|p(\xi) - \hat{p}(\xi)\|$ on the  compact set $\mathcal{Q}_\textrm{train}$.
With this background, we let $\sup_{\xi \in \mathcal{Q}_{\textrm{train,s}}} \|p(\xi) - \hat{p}(\xi)\|_\infty$ denote the supremum norm of the approximation 
error over $\mathcal{Q}_{\textrm{train,s}}$.  

\vspace{-.2cm}
\begin{remark}[\bf\textit{Properties of the Training Set}]
Notice that the set of training data $\mathcal{X}_{\textrm{train}}$ is assumed 
to contain the set of initial conditions $\mathcal{X}_0$. This allows us to 
guarantee that the neural network can be trained over the domain of definition 
of  the Lyapunov function $W$ in Lemma~\ref{lem:converse_plant} 
(see Figure~\ref{fig:sets_illustration} for an illustration).
If, by contrast, the set $\mathcal{X}_{\textrm{train}}$ were contained in  
$\mathcal{X}_0$, then the set of initial of~\eqref{eq:closedloop_perception} 
must be modified so that the trajectories do not leave the set 
$\mathcal{X}_{\textrm{train}}$.  
\QEDB\end{remark}
\vspace{-.2cm}

We begin by characterizing the performance of~\eqref{eq:closedloop_perception} 
when residual networks are utilized to reconstruct the system state. For simplicity of exposition, we outline the main result for the 
case where $n = n_{\xi}$, and we then discuss how to consider the case 
$n < n_{\xi}$ in Remark~\ref{rem:highDimensionSpace}.

\vspace{-.2cm}
\begin{proposition}{\bf \textit{(Transient Performance of 
Algorithm~\ref{alg:perception_cost} with Residual Neural Network)}}
\label{prop:General_perceptionstate}
Consider the closed-loop system~\eqref{eq:closedloop_perception}, let 
Assumptions \ref{as:steadyStateMap}-\ref{as:perceptionmap} be satisfied, and 
assume $n = n_{\xi}$.
Assume that the training set $\mathcal{Q}_{\textrm{train},s}$ is a 
$\varrho$-cover of $\mathcal{Q}_{\textrm{train}}$ with respect to the partial 
order $\preceq$, for some $\varrho > 0$. 
Let $p_{\textrm{resNet}}: \R^{n_\xi} \to R^{n_\xi}$ describe a residual network,
and  assume that it  can be decomposed as $p_{\textrm{resNet}} = m + A$, where 
$m: \R^{n_\xi} \to \R^{n_\xi}$ is monotone and  $A: \R^{n_\xi} \to \R^{n_\xi}$ 
is a linear function. 
If Algorithm~\ref{alg:perception_cost} is implemented with 
$\hat p =  p_\textrm{resNet}$ and $\eta \in (0,\eta^*)$, then the error $z(t) = (x - x^*_t, u - u^*_t)$ of~\eqref{eq:closedloop_perception} satisfies~\eqref{eq:genErrorBound} with $\kappa_1, \kappa_2, \kappa_3$ as in Theorem~\ref{thm:General_stability}, $\gamma = 0$, and 
\begin{align}
\delta & =  \ell_{h_u} \ell_x \sqrt{n_\xi} 
\Big(3  \sup_{\xi \in \mathcal{Q}_{\textrm{train},s}} \| p(\xi) - p_\textrm{resNet}(\xi)\|_\infty \nonumber \\
    & \hspace{3.9cm}+ 2 \, \omega_p(\varrho) + 2  \|A\|_\infty \varrho
 \Big) \label{eq:error_perception}
\end{align}
where $\omega_p$ is a modulus of continuity of $p$ on $\mathcal{Q}_\textrm{train}$. 
\QEDB
\vspace{-.3cm}
\end{proposition}

\begin{proof}
Start by noticing that~\eqref{eq:closedloop_perception_c} can be written in the 
generic form~\eqref{eq:closedloop_error_b} with the error $e(x,u)$ given by 
$e(u,x)  =  H(u)^\top  \nabla \psi(x) -  H(u)^\top \nabla \psi (\hat{x})  =  H(u)^\top  \nabla \psi(p(\xi)) -  H(u)^\top \nabla \psi (\hat{p}(\xi))$, by 
simply adding and subtracting  the true gradient $H(u)^\top  \nabla \psi(x)$ 
in~\eqref{eq:closedloop_perception_c}.  A uniform bound on the norm of $e(u,x)$ 
over the compact set $\mathcal{Q}_\textrm{train}$ is given by:
\begin{align}
\label{eq:statePerceptionResidualNet}
\|e(u,x)\| & \leq  \ell_{h_u} \big[\nabla \psi(p(\xi)) - \nabla \psi  (\hat{p}(\xi)) \big] \nonumber \\ 
& \leq \ell_{h_u} \ell_x  \| p(\xi) - p_{\textrm{resNet}}(\xi)\| \nonumber \\
&  \leq   \ell_{h_u} \ell_x  \sup_{\xi \in \mathcal{Q}_\textrm{train}} \|p(\xi) - p_{\textrm{resNet}}(\xi)\| \, ,
\end{align}
where we have used Assumption~\ref{as:steadyStateMap} and the fact that the norm of the Jacobian of $h$ is bounded over the compact set $\mathcal{U}$. Next, notice first that $\|p(\xi) - \hat{p}(\xi)\| \leq \sqrt{n_\xi} \| p(\xi) - p_\textrm{resNet}(\xi)\|_\infty$. Since $\mathcal{Q}_{\textrm{train},s}$ is a  $\varrho$-cover of $\mathcal{Q}_{d}$ with respect to the partial order $\preceq$, for some $\varrho > 0$, and $p_{\textrm{resNet}} = m + A$, the infimum norm of the estimation  error can be upper bounded as $\sup_{\xi \in \mathcal{Q}_\textrm{train}} \| p(\xi) - p_\textrm{resNet}(\xi)\|_\infty \leq 3 \sup_{\xi \in \mathcal{Q}_{\textrm{train},s}} \| p(\xi) - p_\textrm{resNet}(\xi)\|_\infty + 2 \, \omega_p(\varrho) + 2 \|A_p\|_\infty$ as shown in~\cite[Theorem~7]{Tabuada-pmlr-v144-marchi21a}. The result then follows from Theorem~\ref{thm:General_stability} by setting $\gamma = 0$ and $\delta$ as in \eqref{eq:error_perception}.
\end{proof}

Proposition~\ref{prop:General_perceptionstate} shows that the control method 
in Algorithm~\ref{alg:perception_cost} guarantees convergence to the optimizer 
of~\eqref{opt:objectiveproblem}, up to an error that depends only on 
the uniform approximation error of the adopted neural network.
Notice that 
$\sup_{\xi \in \mathcal{Q}_{\textrm{train},s}} \| p(\xi) - p_\textrm{resNet}(\xi)\|_\infty$ is a constant that denotes the worse-case
approximation error of the training data over the compact set 
$ \mathcal{Q}_{\textrm{train},s}$. 
More precisely, the result characterizes the role of the approximation errors 
due to the use of a neural network in the transient and asymptotic performance 
of the interconnected system~\eqref{eq:closedloop_perception}.

\vspace{-.2cm}

\begin{remark}[\bf \textit{Perception in High-Dimensional Spaces}]
\label{rem:highDimensionSpace}
When $n_\xi > n$, one can consider training a neural network to approximate a 
lifted map $\tilde p: \R^{n_\xi} \to \R^{n_\xi}$ defined as 
$\tilde p = \iota \circ p$, where $\iota: \R^n \to \R^{n_\xi}$ is the injection:
$\iota(x) = (x_1, \ldots, x_n, 0, \ldots, 0)$ for any $x \in \R^n$.
From Assumption~\ref{as:perceptionmap}, it follows that 
$\tilde p(\cdot)$ is such that $\tilde p(\xi) = (x, 0_{q-n} )$ for any $x \in \mathcal{X}_r$, with $\xi$ an observation generated by the map $q(\cdot)$. 

In this case, we use the training set 
$\{(x^{(i)}, 0_{q-n}), \xi^{(i)}\}_{i = 1}^N$ to train the neural network 
implementing a map $p_{\textrm{resNet}}: \R^{n_\xi} \to \R^{n_\xi}$~\cite{Marchi2022}. 
Subsequently, the perception map $\hat p$ that will be used in~\eqref{eq:closedloop_perception_c} is given by 
$\hat p = \pi \circ p_{\textrm{resNet}}$, where $\pi: \R^{n_\xi} \to \R^n$ is a 
projection map that returns the first $n$ entries of its argument, namely,
$\pi(y) = (y_1, \ldots, y_n)$, for any $y \in \R^q$.  
Summarizing, the  training step $\operatorname{NN-learning}(\cdot)$  in 
Algorithm~\ref{alg:perception_cost} in this case involves the training of the 
map $p_{\textrm{resNet}}$, followed by the projection 
$\hat{x} = \hat p(\xi) = \pi(p_{\textrm{ claim}}(\xi))$.  
Finally, we notice that in this case the claim in 
Proposition~\ref{eq:statePerceptionResidualNet} holds unchanged by replacing 
$p(\xi)$ with $\tilde p(\xi)$. This follows by noting that $\|p(\xi) - \hat{p}(\xi)\| \leq \sqrt{n_\xi} \| p(\xi) - p_\textrm{resNet}(\xi)\|_\infty$.
\QEDB
\vspace{-.3cm}
\end{remark}

\begin{remark}[\bf \textit{Density of Training Set}]
Proposition~\ref{prop:General_perceptionstate} requires the training set 
$\mathcal{Q}_{\textrm{train},s}$ is a $\varrho$-cover of 
$\mathcal{Q}_{\textrm{train}}$, with respect to the partial order $\preceq$.
As pointed out in~\cite{Marchi2022}, verifying this condition often involves 
computing the relative position of the training points and the points in the set
$\mathcal{Q}_{\textrm{train}}$. 
When this is not possible, \cite[Lemma~2]{Marchi2022} shows that there exists a 
relationship between $\varrho$-covering of the set  
$\mathcal{Q}_{\textrm{train}}$ with respect to $\preceq$ and the density of the 
training points. In particular, the authors show that if the set of training 
points is a $\varrho'$-cover of $\mathcal{Q}_{\textrm{train}}$, then it is also 
a $\varrho$-covering for a set  the set  $\mathcal{Q}_{\textrm{train}}'$, $\mathcal{Q}_{\textrm{train}} \subset \mathcal{Q}_{\textrm{train}}'$, with respect to $\preceq$, for some $\varrho > \varrho'$; see~\cite[Lemma~2]{Marchi2022}.  
\QEDB\end{remark}

\vspace{-.2cm}

In the remainder of this section, we focus on characterizing the performance 
of~\eqref{eq:closedloop_perception} when a feedforward network is utilized to 
reconstruct the system state. 
More precisely, we consider cases where the training set 
$\{\xi^{(i)}, x^{(i)}\}_{i = 1}^N$ is utilized to train $n$ multilayer 
feedforward networks, each of them implementing a map
$p_{\textrm{feedNet},i}: \R^{n_\xi} \to \R$ that estimates the $i$-th component 
of the system state $\hat x_i = p_{\textrm{feedNet},i}(\xi)$. 
In this case, we  assume that Algorithm~\ref{alg:perception_cost} is implemented
with  $\hat p = p_{\textrm{feedNet}}$, where $p_{\textrm{feedNet}}(
\xi) := (p_{\textrm{feedNet},1}(
\xi), \ldots, p_{\textrm{feedNet},n}(
\xi))$ in~\eqref{eq:closedloop_perception}. 
Next, we recall that feedforward neural networks are capable of approximating 
any measurable function on compact sets with any desired degree of accuracy 
(see, for instance,~\cite{hornik1989multilayer,barron1994approximation} and the 
bounds in~\cite{mehrabi2018bounds,goebbels2022sharpness}). 

\vspace{-.2cm}

\begin{proposition}{\bf \textit{(Transient Performance of 
Algorithm~\ref{alg:perception_cost} with Feedforward Neural Network)}}
\label{prop:General_perceptionstate_forward}
Consider the closed-loop system~\eqref{eq:closedloop_perception} and let Assumptions \ref{as:steadyStateMap}-\ref{as:perceptionmap} be satisfied. Suppose that $\hat p(
\xi) = (p_{\textrm{feedNet},1}(
\xi), \ldots, p_{\textrm{feedNet},n}(
\xi))$, with $p_{\textrm{feedNet},i}$ approximating the map $p_i(\xi)$.  Then, if $\eta \in (0,\eta^*)$, the error $z(t) = (x - x^*_t, u - u^*_t)$ of~\eqref{eq:closedloop_perception} satisfies~\eqref{eq:genErrorBound} with $\gamma = 0$, 
\begin{align}
\delta & = \ell_{h_u} \ell_x \sqrt{n} \sup_{\xi \in \mathcal{Q}_{\textrm{train}}} \|p(\xi) - p_{\textrm{feedNet}}(\xi)\|_\infty \, , \label{eq:error_perception_fnet}
\end{align}
and $\kappa_1, \kappa_2, \kappa_3$ as in Theorem~\ref{thm:General_stability}. 
\QEDB
\end{proposition}

The proof follows similar steps as in 
Proposition~\ref{prop:General_perceptionstate}, and it is omitted. 
Proposition~\ref{prop:General_perceptionstate_forward} shows that the control 
method in Algorithm~\ref{alg:perception_cost} guarantees convergence to the 
optimizer of~\eqref{opt:objectiveproblem}, up to an error that depends only on 
the uniform approximation error
$\sup_{\xi \in \mathcal{Q}_{\textrm{train}}} \|p(\xi) - p_{\textrm{feedNet}}(\xi)\|_\infty$ (computed over the entire training set 
$\mathcal{Q}_{\textrm{train}}$).
Notice that, with respect to Proposition~\ref{prop:General_perceptionstate}, 
the adoption of a feedforward network allows us to provide tighter 
guarantees in terms of the entire set $\mathcal{Q}_{\textrm{train}}$ 
(as opposed to the set of available samples 
$\mathcal{Q}_{\textrm{train},s}$).  
We conclude by noting that the bound~\eqref{eq:error_perception_fnet} can be 
further customized for specific error bounds, given the architecture of the 
feedforward network~\cite{mehrabi2018bounds,goebbels2022sharpness}.  

\begin{remark}[\bf \textit{Noisy generative and perception maps}]
Assumption~\ref{as:perceptionmap} is borrowed from~\cite{dean2021certainty} and it holds when, for example, $q$ is injective. Although the model in Assumption~\ref{as:perceptionmap} is used for simplicity, the subsequent analysis of our perception-based controllers can be readily extended to the case where: (i) the perception map imperfectly estimates the state; that is, one has that $p(\xi) = x + \nu$, with $\xi = q(x)$, and where $\nu \in \mathbb{R}^n$ is a bounded error~\cite{dean2020robust}. (ii) When unknown externalities enter the generative map.  One way to collectively account for both externalities entering $q$ and for approximate perception map is to use the noisy model $p(q(x)) = x + \nu'$, with $\nu' \in \mathbb{R}^n$ a given error (bounded in norm). The results presented in this section can be readily modified to account for this additional error by adding a term proportional to the norm of $\nu'$ in the parameter $\delta$.
\end{remark}

\section{Optimization with Cost-Function Perception}
\label{sec:costPerception}

In this section, we propose an algorithm to address 
Problem~\ref{pro:cost}, and we tailor the conclusions drawn in  
Theorem~\ref{thm:General_stability} to characterize the performance of the 
proposed algorithm.

\subsection{Algorithm Description}
To determine estimates of the gradient functions 
$\nabla \phi(u), \nabla\psi(x)$, we assume the availability of functional 
evaluations $\{(u^{(i)}, \phi(u^{(i)}))\}_{i = 1}^{N}$ and 
$\{(x^{(i)},  \psi(x^{(i)}))\}_{i = 1}^M$, with 
$u^{(i)} \in \mc C$ and 
$x^{(i)} \in \mathcal{X}_{\textrm{train}}$.  
We then consider the training of two neural networks that approximate the  
functions $u \mapsto \phi(u)$ and $x \mapsto \psi(x)$, respectively, to 
determine $\hat \phi(u)$, $\hat \psi(x)$.
Accordingly, we modify the controller~\eqref{eq:ideal_controller} to operate 
with estimates of the system state $\hat x$ produced by the neural network. The 
proposed framework is described  in Algorithm~\ref{alg:cost_perception} and 
illustrated in Figure~\ref{fig:system_2}. 

\begin{algorithm}
\caption{Optimization with NN Cost Perception}
\label{alg:cost_perception}

\# \textbf{Training}

\quad Given: $\{u^{(i)},  \phi(u^{(i)})\}_{i = 1}^{N}$, $\{x^{(i)}, \psi(x^{(i)})\}_{i = 1}^M$ 

\quad Obtain: 
\vspace{-.2cm}
\begin{align*}
\hat \phi & \gets \operatorname{NN-learning}(\{(u^{(i)}, \phi(u^{(i)}))\}_{i = 1}^N)    \\ 
\hat \psi & \gets \operatorname{NN-learning}(\{(x^{(i)}, \psi(x^{(i)}))\}_{i = 1}^M)    
\end{align*}

\# \textbf{Gradient-based  Feedback Control}

\quad Given: $x(t_0) \in \mathcal{D}_0$, $u(t_0) \in \mathcal{C}$, NN maps $\hat \phi, \hat \psi$.  

\quad For $t \geq t_0$:
\begin{subequations}
\label{eq:closedloop_cost}
\begin{align}
 \dot x & =  f(x, u, w_t)  \\
      \hat{g}_u(u) & = \sum_{i = 1}^{n_u} \frac{1}{2\varepsilon} \left( \hat \phi(u + \varepsilon b_i) -  \hat \phi(u - \varepsilon b_i) \right) b_i , \label{eq:closedloop_cost_u} \\
      \hat{g}_x(x) & = \sum_{i = 1}^n \frac{1}{2\varepsilon} \left( \hat \psi(x + \varepsilon d_i) -  \hat \psi(x - \varepsilon d_i) \right) d_i ,  \\
      \dot u & = \proj \hspace{-.1cm} \left\{u-\eta \left(\hat g_u(u) + H(u)^\top \hat g_x(x) \right) \right\}-u \label{eq:closedloop_cost_contr} \hspace{-.2cm}
\end{align}
\end{subequations}

\end{algorithm}

\begin{figure}[h!]
\centering \includegraphics[width=1.0\columnwidth]{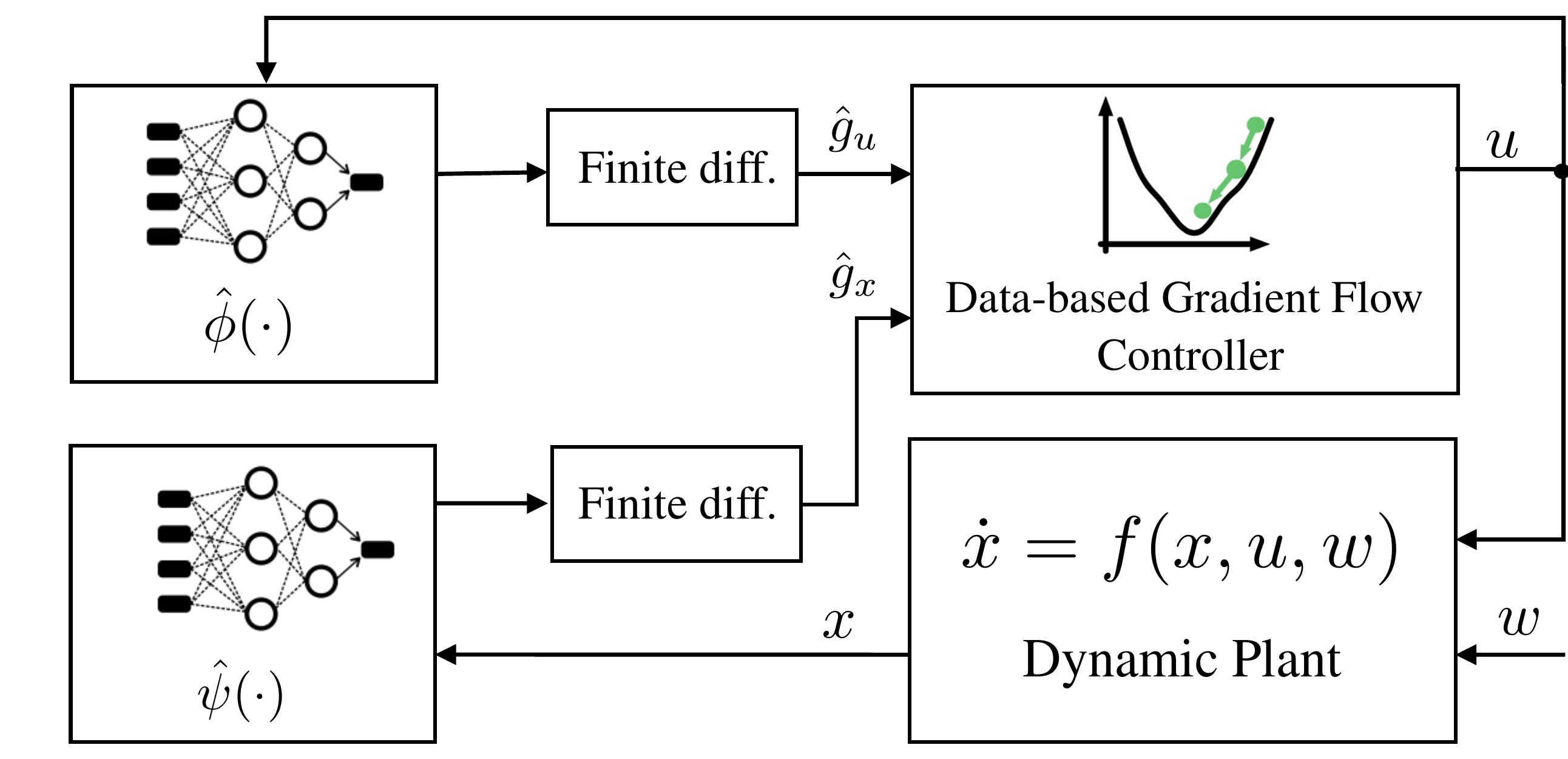}
\caption{Schematics of the control framework proposed in 
Algorithm~\ref{alg:cost_perception}. A gradient-based controller 
is utilized to control the system to the optimizer of an unknown 
cost function. The underlying cost function is estimated from 
data samples through a residual neural network.}
\label{fig:system_2}
\end{figure}

In Algorithm~\ref{alg:cost_perception}, the gradients of the costs are obtained 
via centered difference, applied to the estimated maps $\hat{\phi}$ and 
$\hat{\psi}$, where $\varepsilon > 0$, $b_i$ denotes the $i$-th canonical vector
of $\R^{n_u}$, and $d_i$ is the the $i$-th canonical vector of $\R^n$. The 
computation of the approximate gradient $\hat{g}_u$ (respectively, $\hat{g}_x$) 
thus requires $2 n_u$ functional evaluations (respectively, $2 n$) of the neural
network map $\hat{\phi}$ (respectively, $\hat \psi$). The gradient estimates 
$\hat{g}_u$ and $\hat{g}_x$ are then utilized in the gradient-based feedback 
controller~\eqref{eq:closedloop_cost_contr}.  

\subsection{Analysis of Algorithm~\ref{alg:cost_perception}}
We begin by characterizing the performance of~\eqref{eq:closedloop_cost} 
when feedforward networks are utilized to estimate the costs.~ 

\vspace{-.2cm} 

\begin{proposition}{\bf \textit{(Transient Performance of 
Algorithm~\ref{alg:cost_perception} with Feedforward Neural Network)}}
\label{prop:General_perceptioncost_forward}
Suppose that feedforward network maps  $\hat{\phi}_{\textrm{feedNet}}$ and $\hat{\psi}_{\textrm{feedNet}}$ approximate the costs $\phi$ and $\phi$ over the compact sets $\mathcal{X}_{\textrm{train}}$ and $\mathcal{C}_{\textrm{train}} := \mathcal{C} + \mathcal{B}[\varepsilon]$, respectively.  Consider the interconnected system~\eqref{eq:closedloop_cost}, with $\hat \phi = \hat{\phi}_{\textrm{feedNet}}$ and $\hat \psi = \hat{\psi}_{\textrm{feedNet}}$, and let Assumptions \ref{as:steadyStateMap}-\ref{as:mapu} be satisfied. If $\eta \in (0,\eta^*)$, then error $z(t) = (x - x^*_t, u - u^*_t)$ satisfies~\eqref{eq:genErrorBound} with $\kappa_1, \kappa_2, \kappa_3$ as in Theorem~\ref{thm:General_stability}, $\gamma = 0$, and 
\begin{align}
    \delta & = e_{u,\textrm{fd}} + n_u  \varepsilon^{-1} \sup_{u \in \mathcal{C}_{\textrm{train}} } |\phi(u) - \hat{\phi}_{\textrm{feedNet}}(u)| \nonumber \\
    & ~~+ \ell_{h_u} e_{x,\textrm{fd}} + n \varepsilon^{-1} \ell_{h_u} \sup_{u \in \mathcal{X}_{\textrm{train}}} |\psi(x) - \hat{\psi}_{\textrm{feedNet}}(x)|,  \label{eq:error_cost_fnet}
\end{align}
where $e_{x,\textrm{fd}}$ and $e_{x,\textrm{fd}}$ are bounds on the centered difference approximation error for the functions $\phi$ and $\psi$, respectively; namely, $ e_{u,\textrm{fd}} = \sup_{u \in \mathcal{C}_{\textrm{train}}} \|\nabla \phi(u) - \sum_{i = 1}^{n_u} \frac{1}{2\varepsilon} \left(\phi(u + \varepsilon b_i) -  \phi(u - \varepsilon b_i) \right) b_i\|$ and $ e_{x,\textrm{fd}} = \sup_{x \in \mathcal{X}_{\textrm{train}}} \|\nabla \psi(x) - \sum_{i = 1}^{n} \frac{1}{2\varepsilon} \left(\psi(x + \varepsilon d_i) -  \psi(x - \varepsilon d_i) \right) d_i\|$.
\QEDB
\vspace{-.3cm}
\end{proposition}

\begin{proof}
We start rewriting~\eqref{eq:closedloop_cost_contr} as~\eqref{eq:closedloop_error_b} by setting $e(x,u) = \nabla \phi(u) - \hat{g}_u(u) + H^\top(\nabla \psi(x) - \hat{g}_x(x))$. Let $g_u(u) := \sum_{i = 1}^{n_u} \frac{1}{2\varepsilon} \left(\phi(u + \varepsilon b_i) -  \phi(u - \varepsilon b_i) \right) b_i$ and $g_x(x) := \sum_{i = 1}^{n} \frac{1}{2\varepsilon} \left(\psi(x + \varepsilon d_i) -  \psi(x - \varepsilon d_i) \right) d_i$ be the finite difference approximations of the true gradients for brevity. Adding and subtracting $g_u(u)$ and $g_x(x)$ using the triangle inequality, and Assumption~\ref{as:steadyStateMap}, we get $\|e(x,u)\| \leq \|\nabla \phi(u) - g_u(u)\| + \|g_u(u) - \hat g_u(u)\| +  \ell_{h_u} \|\nabla \psi(x) - g_x(x)\| + \ell_{h_u} \|g_x(x) - \hat g_x(x)\|$. The terms $\|\nabla \phi(u) - g_u(u)\|$ and $\|\nabla \psi(x) - g_x(x)\|$ are errors due to a finite difference approximation of the true gradients, and are bounded by $e_{u,\textrm{fd}}$ and $e_{x,\textrm{fd}}$, respectively. On the other hand, $\|g_u(u) - \hat g_u(u)\|$ can be bounded as: 
\begin{align}
    & \|g_u(u) - \hat g_u(u)\|  = \Big\| \sum_{i = 1}^{n_u} \frac{1}{2\varepsilon} \left(\phi(u + \varepsilon b_i) -  \phi(u - \varepsilon b_i) \right) b_i \nonumber \\
    & ~~~~~ - \sum_{i = 1}^{n_u} \frac{1}{2\varepsilon} \left(\hat{\phi}_{\textrm{feedNet}}(u + \varepsilon b_i) -  \hat{\phi}_{\textrm{feedNet}}(u - \varepsilon b_i) \right) b_i \Big\|\nonumber \\
    & ~~ \leq  \sum_{i = 1}^{n_u} \frac{1}{2\varepsilon} |\phi(u + \varepsilon b_i) - \hat{\phi}_{\textrm{feedNet}}(u + \varepsilon b_i)| \|b_i\| \nonumber \\
    &~~~~~ + \sum_{i = 1}^{n_u} \frac{1}{2\varepsilon} | \hat{\phi}_{\textrm{feedNet}}(u - \varepsilon b_i) - \phi(u - \varepsilon b_i)| \|b_i\| \nonumber \\
    & ~~ \leq \frac{n_u}{2\varepsilon}  \sup_{u \in \mathcal{C} + \mathcal{B}[\varepsilon]} |\phi(u) - \hat{\phi}_{\textrm{feedNet}}(u)| 
\end{align}
where we used the fact that $\|b_i\| = 1$. Similar steps can be used to bound the error term $\|g_x(x) - \hat g_x(x)\|$ to get the final expression for $\delta$ in~\eqref{eq:error_cost_fnet}.  
\end{proof}

Proposition~\ref{prop:General_perceptioncost_forward} shows that the control 
method in Algorithm~\ref{alg:cost_perception} guarantees convergence to the 
optimizer 
of~\eqref{opt:objectiveproblem}, up to an error that depends on the uniform 
approximation error of the neural networks and on the accuracy of the centered 
approximation method.
More precisely, the result characterizes the role of the approximation errors 
due to the use of a feedforward neural network in the transient and asymptotic performance 
of the interconnected system~\eqref{eq:closedloop_cost}.

In the remainder of this section, we focus on characterizing the performance 
of~\eqref{eq:closedloop_cost}  when a residual network is utilized to 
reconstruct the system state.
To provide guarantees for residual networks, it is necessary to replace 
$\phi(\cdot)$ by its the lifted counterpart 
$\tilde \phi: \R^{n_u} \to \R^{n_u}$, defined as $\tilde \phi = \iota_\phi \circ \phi$, where $\iota_\phi: \R \to \R^{n_u}$ is the injection $\iota_\phi(z) = (z, 0, \ldots, 0)$ for any $z \in \R$. Following~\cite{Tabuada-pmlr-v144-marchi21a}, we consider a residual network map $\hat{\phi}_{\textrm{resNet}}: \R^{n_u} \to \R^{n_u}$ approximating the lifted map $\tilde \phi$; the function $\hat \phi$ used in~\eqref{eq:closedloop_cost} is then given by $\hat \phi(u) = \hat{\phi}_{\textrm{resNet}}(u)^\top b_1$, where we recall that $b_1$ is the first canonical vector of $\R^{n_u}$. Similarly, consider the lifted map $\tilde \psi: \R^{n} \to \R^{n}$ defined as $\tilde \psi = \iota_\psi \circ \psi$, where $\iota_\phi: \R \to \R^{n}$ is such that $\iota_\psi(z) = (z, 0, \ldots, 0)$ for any $z \in \R$, and consider a residual network map $\hat{\psi}_{\textrm{resNet}}: \R^{n} \to \R^{n}$ approximating the lifted map $\tilde \psi$. Accordingly, it follows that $\hat \psi(x) = \hat{\psi}_{\textrm{resNet}}(x)^\top d_1$. With this setup, we have the~following. 

\vspace{-.2cm} 

\begin{proposition}{\bf \textit{(Transient Performance of 
Algorithm~\ref{alg:cost_perception} with Residual Neural Network)}}
\label{prop:General_perceptioncost_residual}
Suppose that residual network maps  $\hat{\phi}_{\textrm{resNet}}$ and $\hat{\psi}_{\textrm{resNet}}$ approximate the functions $\tilde \phi$ and $\tilde \phi$ over the compact sets $\mathcal{X}_{\textrm{train}}$ and $\mathcal{C}_{\textrm{train}}$, respectively.
Suppose that the set of training points $\sbs{\mc C}{train,s} := \{u_i^{(i)}\}$ is a $\varrho_u$-cover of $\mathcal{C}_{\textrm{train}}$ with respect to the partial order $\preceq$, for some $\varrho_u > 0$, and  $\sbs{\mc X}{train,s} =  \{x_i^{(i)}\}$ is a $\varrho_x$-cover of $\mathcal{X}_{\textrm{train}}$ with respect to the partial order $\preceq$, for some $\varrho_x > 0$. Moreover, suppose  that  the residual network maps can be decomposed as $\hat \phi_{\textrm{resNet}} = m_u + A_u$ and $\hat \psi_{\textrm{resNet}} = m_x + A_x$,  where $m_u: \R^{n_u} \to \R^{n_u}$ and $m_x: \R^{n} \to \R^{n}$ are monotone, and  $A_u, A_x$ are a linear functions. Consider the interconnected system~\eqref{eq:closedloop_cost}, with $\hat \phi(u) = \hat{\phi}_{\textrm{resNet}}(u)^\top b_1$ and $\hat \psi(x) = \hat{\psi}_{\textrm{resNet}}(x)^\top d_1$, and let Assumptions \ref{as:steadyStateMap}-\ref{as:mapu} be satisfied. If $\eta \in (0,\eta^*)$, the error $z(t) = (x - x^*_t, u - u^*_t)$ satisfies~\eqref{eq:genErrorBound} with $\kappa_1, \kappa_2, \kappa_3$ as in Theorem~\ref{thm:General_stability}, $\gamma = 0$, and 
\begin{align}
    \delta & = e_{u,\textrm{fd}} + n_u^{3/2}  \varepsilon^{-1} \left( 3 e_{u,\textrm{train}} + 2 \, \omega_\phi(\varrho_u) + 2  \|A_u\|_\infty \right)\nonumber \\
    & ~~~+ \ell_{h_u} e_{x,\textrm{fd}} + n^{3/2} \varepsilon^{-1} \ell_{h_u} \left( 3 e_{x,\textrm{train}} + 2 \, \omega_\psi(\varrho_x) + 2  \|A_x\|_\infty \right) \label{eq:error_cost_fnet}
\end{align}
where $e_{x,\textrm{fd}}$ and $e_{x,\textrm{fd}}$ are defined as in Proposition~\ref{prop:General_perceptioncost_forward}, 
$\omega_u, \omega_x$ are the moduli of continuity of $\tilde \phi$ and $\tilde \psi$, respectively, and 
\begin{align*}
 e_{u, \textrm{train}} &:= \sup_{u \in \sbs{\mc C}{train,s} } \|\tilde \phi(u) - \hat{\phi}_{\textrm{feedNet}}(u)\|_\infty,\\   
e_{x, \textrm{train}} &:= \sup_{x \in \sbs{\mc X}{train,s} } \|\tilde \psi(x) - \hat{\psi}_{\textrm{feedNet}}(x)\|_\infty.
\end{align*}
\vspace{-.3cm}
\QEDB\end{proposition}
The proof of Proposition~\ref{prop:General_perceptioncost_residual} follows 
similar steps as the proof of   
Proposition~\ref{prop:General_perceptioncost_forward} and is omitted for space 
limitations.
Proposition~\ref{prop:General_perceptioncost_residual} shows that the control 
method in Algorithm~\ref{alg:cost_perception} guarantees convergence to the 
optimizer of~\eqref{opt:objectiveproblem}, up to an error that depends on the 
uniform approximation error of the neural networks and on the accuracy of the 
centered approximation method.
Notice that, with respect to the characterization 
in~Proposition~\ref{prop:General_perceptioncost_forward}, the use of a residual 
neural network allows us to characterize the error with respect to accuracy of 
the available set of samples $\sbs{\mc C}{train,s}$ and $\sbs{\mc X}{train,s}$.

\section{Application to robotic control}

In this section, we illustrate how to apply the proposed framework to control 
a unicycle robot to track an optimal equilibrium point and whose position is 
accessible only through camera images. We consider a robot described by 
unicycle dynamics  with state $x = (a,b, \theta)\in \real^3,$ where 
$r := (a,b)^\tsp \in \real^2$ denotes the position of the robot in a 
2-dimensional plane, and $\theta \in(-\pi, \pi]$ denotes its orientation with 
respect to the $a-$axis~\cite{terpin2021distributed}. The unicycle dynamics are:
\begin{align}\label{eq:unicycle}
\dot a &= v \cos (\theta), &
\dot b &= v \sin (\theta), &
\dot \theta &= \omega,
\end{align}
where $v,\omega \in \real$ are the controllable inputs.
Given the dynamics~\eqref{eq:unicycle}, we assume that its state 
$x = (a,b,\theta)$ is not directly measurable for control purposes, instead,
at every time $x$ can be observed only through a noisy camera image, denoted 
by $\xi = q(x)$; see, e.g.~\cite{dean2021certainty,dean2020robust}.
To tackle the desired problem, we consider an instance 
of~\eqref{opt:objectiveproblem} with:
\begin{align*}
\phi(u) &= \|u\|^2, & 
\psi(r) &= \norm{r - r^{f}}^2,
\end{align*}
where $r^{f} \in \real^2$ denotes the desired final position of the robot.

To address this problem, we consider a two-level control architecture, where 
an onboard (low-level) stabilizing controller is first used to stabilize the 
unicycle dynamics~\eqref{eq:unicycle} (to guarantee satisfaction of 
Assumption~\ref{as:stabilityPlant}) and, subsequently, the control framework 
outlined in Algorithm~\ref{alg:perception_cost} is utilized to select an 
optimal high-level control references. To design a stabilizing controller, let
$u=(u_a, u_b) \in \real^2$ denote the instantaneous high-level control input, 
and consider the standard change of variables from rectangular to polar 
coordinates $(\xi, \phi)$, given by: 
\begin{align}\label{eq:polarCoordinates}
\xi &:= \norm{u - x},  & 
\phi &:= \operatorname{atan}\left(\frac{u_b - b}{u_a-a} \right) - \theta.
\end{align}
In the new variables, the dynamics read as:
\begin{align}\label{eq:dynamics-xi-phi}
\dot \xi &= -v \cos(\phi), & 
\dot \phi &= \frac{v}{\xi} \sin(\phi) - \omega.
\end{align}
The following lemma provides a stabilizing control law 
for~\eqref{eq:dynamics-xi-phi}.

\begin{lemma}{\bf \textit{(Stability of unicycle dynamics)}}
\label{lem:stabilizingControlUnicycle}
The unicycle dynamics~\eqref{eq:dynamics-xi-phi} with the following control 
law:
\begin{align}\label{eq:stabilizingControl}
v &= k \xi \cos(\phi), &
\omega  &= k (\cos(\phi) +1)\sin(\phi) + k\phi,
\end{align}
$k>0$, admit a unique equilibrium point $(\xi, \phi) = (0,0)$ that is globally
exponentially stable. 
\end{lemma}
\begin{proof}
Noticing that the dynamics of $\xi$ and $\phi$ are decoupled, consider the 
Lyapunov function $V(\phi) = \frac{1}{2} \phi^2$. We have:
\begin{align*}
\dot V(\phi) = - k \phi\sin(\phi) - k\phi^2 \leq -k \phi^2 = -2k V(\phi),
\end{align*}
with $V(\phi)=0$ if and only if $\phi=0$. The proof of exponential stability 
of $\xi$ follows immediately from~\cite[Lem.~2.1]{terpin2021distributed}.
\end{proof}

According to Lemma~\ref{lem:stabilizingControlUnicycle}, the 
dynamics~\eqref{eq:dynamics-xi-phi} with the on-board control 
law~\eqref{eq:stabilizingControl} satisfy Assumption~\ref{as:stabilityPlant}.
We next apply the perception-based control framework outlined in Algorithm~\ref{alg:cost_perception} to design the reference input $u$ to be 
utilized in~\eqref{eq:polarCoordinates}.

\begin{figure}[t]
\centering
\subfigure[]{\includegraphics[width=1.0\columnwidth]{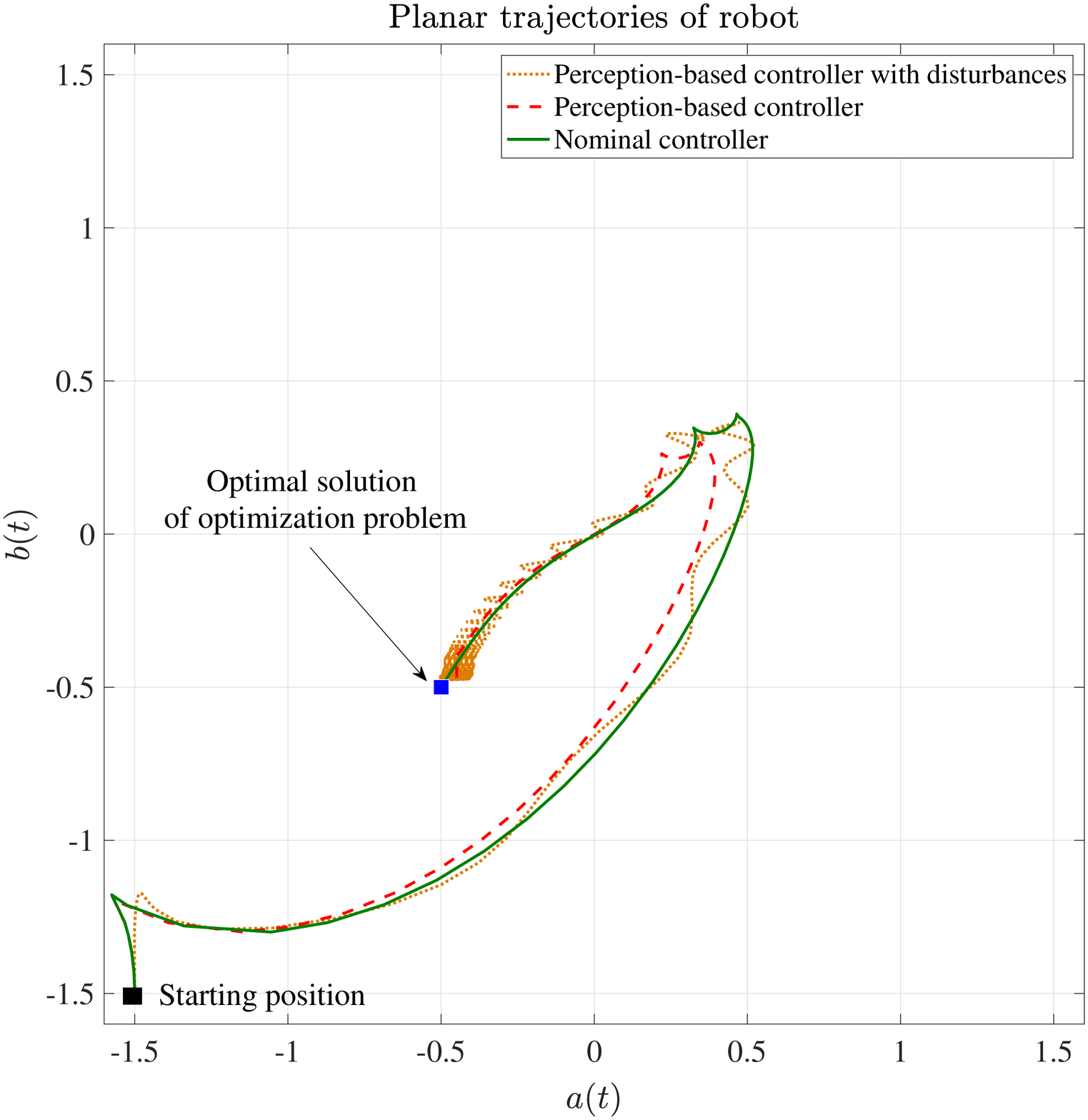}} \vspace{-.2cm} \\
\centering
\subfigure[]{\includegraphics[width=.85\columnwidth]{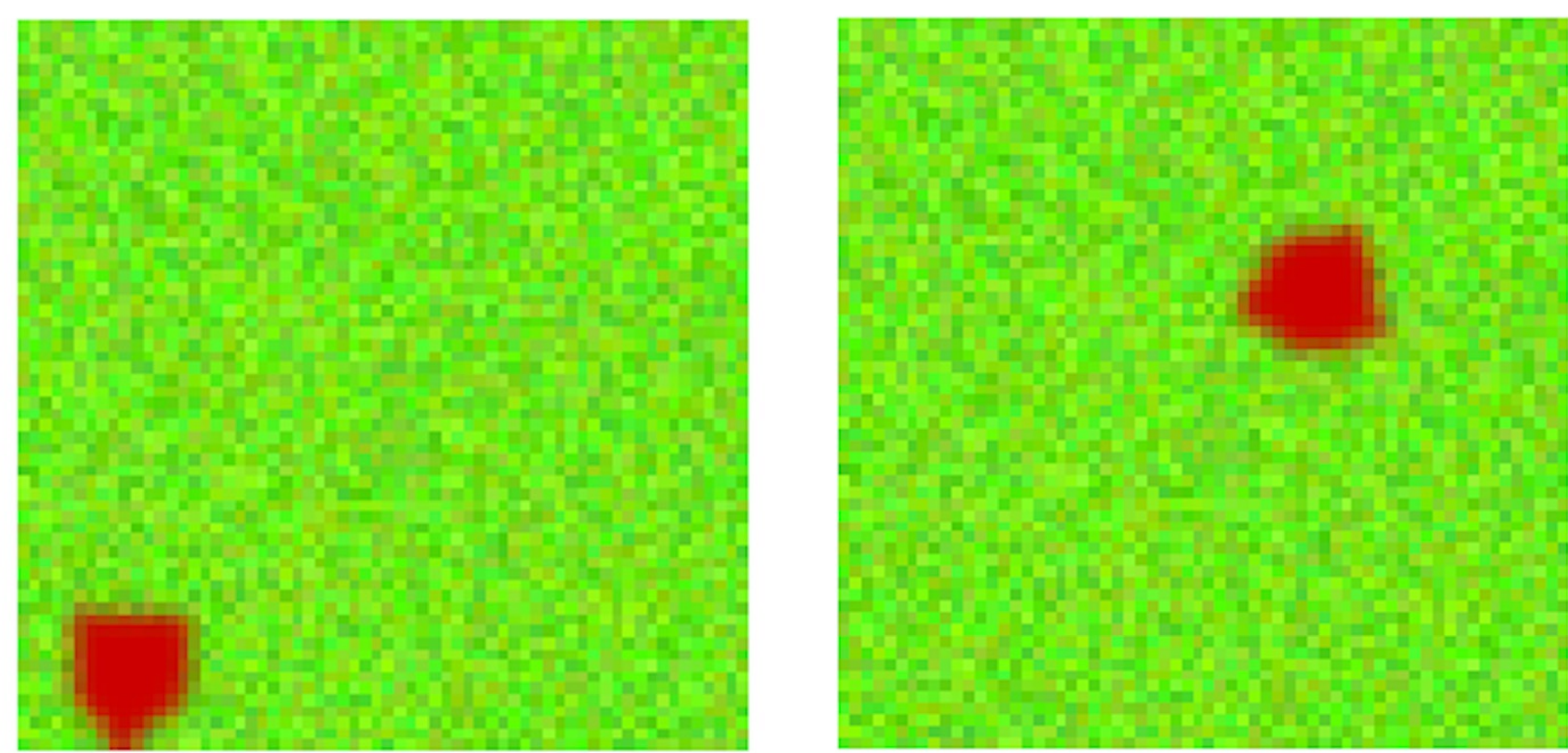}}   \\
\caption{Numerical validation of Algorithm~\ref{alg:perception_cost}. The 
technique is used to control a robot whose state $(a,b,\theta)$ cannot be 
directly measured and is instead estimated using a residual neural network from aerial images.  
(a) Trajectory of the robot (in the coordinates $(a,b)$) with the perception-based controller and with a nominal controller using perfect state information. (b) Images that the perception-based controller uses at $t=4.32$ seconds (left image)  and $t=70.36$ seconds (right image). 
}
\label{fig:unicyclePerception}
\end{figure}

For our perception-based controller, we  use a residual neural network to estimate the  state perception map; in particular, the neural network returns estimated state coordinates from aerial images following the procedure of Algorithm~\ref{alg:cost_perception}. To prepare for the execution of Algorithm~\ref{alg:cost_perception}, we generate $122,880$ square, blurry, RGB images of size $64 \times 64$ pixels to represent the location and orientation of the robot on a two-dimensional plane. These images were built using the \textit{MATLAB Image Processing Toolbox}. First, we build a base image of a maroon pentagon representing the robot over a grassy background (using ranges of RGB values for different green colors), and then add blurriness to the image using the MATLAB function  \textit{imgaussfilt}; see, for example, the sample images in Figure~\ref{fig:unicyclePerception}(b). This function filters the base image with a 2-D Gaussian smoothing kernel with a standard deviation of $0.5$. 
For our residual network, we used the 
    \textit{resnet50}
 network structure given in the 
    \textit{MATLAB Deep Learning Toolbox}
 and tailored the input output sizes for our particular setting. Specifically, we set the input layer to read in images of size $64 \times 64$ pixels and set the number of possible output labels as $64^2 = 4996$ to account for all possible locations or coordinates of the pixelized image. The number of training images for the network was chosen so that all possible locations or coordinates of the pixelized image are generated for 30 different orientations, totaling $122,880$ training images. After training, the residual network returns estimated locations in the 2D plane by identifying the center pixel of the robot. Of course, during the time-evolution of the closed loop dynamics with the residual network, the returned labels are converted into $(a,b)$ coordinates. 
 
 Simulation results  are presented in Figure~\ref{fig:unicyclePerception}, when using the initial conditions $(a_0, b_0, \theta_0, u_{a_0}, u_{b_0}) = (-1.5, -1.5, -\pi /2, 1, 1)$; moreover, the optimal solution of the problem is $r^* = (-0.5, -0.5)$. As evidenced in Figure ~\ref{fig:unicyclePerception}(a), differences persist between the trajectories produced by the nominal controller, which utilizes perfect state information, and by the perception-based controller. Importantly, imperfect estimates of $(a,b)$ are also due to the fact that the labels returned from our trained network correspond to the pixels of the image, which can limit how well the $(a,b)$ values are represented. Figure~\ref{fig:unicyclePerception}(b) shows sample images at times $t=4.32$ and $t=70.36$; these images are used as inputs to the neural network. 
 As expected, the ideal controller converges to the reference state arbitrarily well, while the
perception-based controller converges to the reference state to within a neighborhood dependent on the error associated with the perceived state. Overall, these simulations demonstrate the claims made in Proposition \ref{prop:General_perceptionstate}.

\section{Application to epidemic control}
\label{sec:results}
In this section, we apply the proposed framework to regulate the 
spread of an epidemic by controlling contact-relevant transmissions. 
The latter is achieved by e.g. selecting the intensity of 
restrictions such as mask-wearing, social restrictions, school 
closures, and stay-at-home orders, etc.
To describe the epidemic evolution, we adopt a 
susceptible-infected-susceptible (SIS) model \cite{AF-AI-GS-JT:07}, 
described by:
\begin{align}
\label{eq:SISscalar}
\dot s &= \mu - \mu s - u \beta s x + \gamma x,  &
\dot x &= u \beta s x - (\gamma + \mu) x, 
\end{align}
where $s\in \real ,x \in \real$ describe the fraction of susceptible 
and infected population, respectively, with $s+x=1$ at all times, 
$u \in (0,1]$ is an input modeling the reduction in contact-relevant 
transmissions, $\beta>0$ is the transmission rate,  $\mu>0$ is the 
death/birth rate, and $\gamma>0$ is the recovery rate. 
Model parameters of~\eqref{eq:SISscalar} are chosen as follows: 
$\beta=4$, $\gamma=1/9$, $\mu=10^{-4}$, cost function parameters: 
$\sps{u}{ref}=0.36$, $\sps{x}{ref}=0.85$, $w_\psi=w_\psi=1$.
As characterized in \cite[Thm. 2.1 and Lem. 2.1]{AF-AI-GS-JT:07}, 
\eqref{eq:SISscalar} admits a unique (unstable) equilibrium point 
described by $x=0$ (called disease-free equilibrium) and a unique 
equilibrium point with $x\neq 0$ (called endemic equilibrium) that is 
exponentially stable~\cite[Thm. 2.4]{AF-AI-GS-JT:07}, thus satisfying 
Assumption~\ref{as:stabilityPlant}. 
We utilize the control problem~\eqref{opt:objectiveproblem} to determine an 
optimal balance between a desired fraction of infections $\sps{x}{ref}$ and 
a desired level of restrictions $\sps{u}{ref}$. More formally, we consider an 
instance of~\eqref{opt:objectiveproblem} with 
$\phi(u) = w_\phi (u-\sps{u}{ref})^2$ 
and $\psi(x) = w_\psi(x-\sps{x}{ref})^2$, where 
$\sps{u}{ref}, \sps{x}{ref} \in [0,1]$ are desired reference inputs 
and states and $w_\phi, w_\psi \in \realnneg$ are weighting factors.
For our simulations, we perform the change of variables 
$(\tilde x, \tilde u)=(x, \frac{1}{u})$; in the new variables, 
$h(\tilde u)=1-\frac{\mu + \gamma}{\beta}\tilde u$ satisfies 
Assumptions~\ref{as:steadyStateMap} and \ref{as:LipsConvex}. 
In order to illustrate the effects of perception on the control loops, 
in what follows we assume $w_t=0$ so all the tracking error can be 
associated to perception errors.

\subsection{Optimization with State Perception}
We begin by illustrating the case of state perception 
(Section~\ref{sec:statePerception}). 
One of the main challenges in predicting and controlling the outbreak of an 
epidemic is related to the difficulty in estimating the true number of 
infections from incomplete information describing documented infected 
individuals with symptoms severe enough to be confirmed. 
During the outbreak of COVID-19, researchers have proposed several 
methods to overcome these challenges and by using several sources of data 
including \cite{armstrong2021identifying} detected cases, recovered cases, 
deaths, test positivity~\cite{chiu2021using}, and mobility 
data~\cite{li2020substantial,GB-ED-JP-AB:21-scirep}.
In our simulations, we adopted the approach in 
Algorithm~\ref{alg:perception_cost} to achieve this task.
For the purpose of illustration, we utilized a map $q(\cdot)$ composed of a set 
of $4$ Gaussian basis functions with mean $\sbs{\mu}{b} = (1, 5, 9, 13)$ and 
variance $\sigma= I$ to  determine the perception signal $\xi$.
The training phase of Algorithm~\ref{alg:perception_cost} has then been 
performed using a feedforward neural network to determine the map 
$\hat x = \hat p(\xi)$ to reconstruct the state of~\eqref{eq:SISscalar}. 
Simulation results are illustrated in Figure~\ref{fig:SISstatePerception}.
As illustrated by the state trajectories in 
Figure~\ref{fig:SISstatePerception}(a)-(b), the use of a neural network 
originates discrepancies between the true state and the estimated state, 
especially during  transient phases. Figure~\ref{fig:SISstatePerception}(c) 
provides a comparison of the tracking error between the use of the ideal 
controller~\eqref{eq:ideal_controller} (which uses the exact state) and the 
controller in Algorithm~\ref{alg:perception_cost}. The numerics illustrate that 
while the exact controller~\eqref{eq:ideal_controller} is capable of converging 
to the desired optimizer with arbitrary accuracy, the controller in 
Algorithm~\ref{alg:perception_cost} yields an error of the order $10^{-2}$,
due to uncertainties in the reconstructed system state. Overall, the simulations
validate the convergence claim made in 
Proposition~\ref{prop:General_perceptionstate_forward}.

\begin{figure}[t]
\centering
\subfigure[]{\includegraphics[width=1.0\columnwidth]{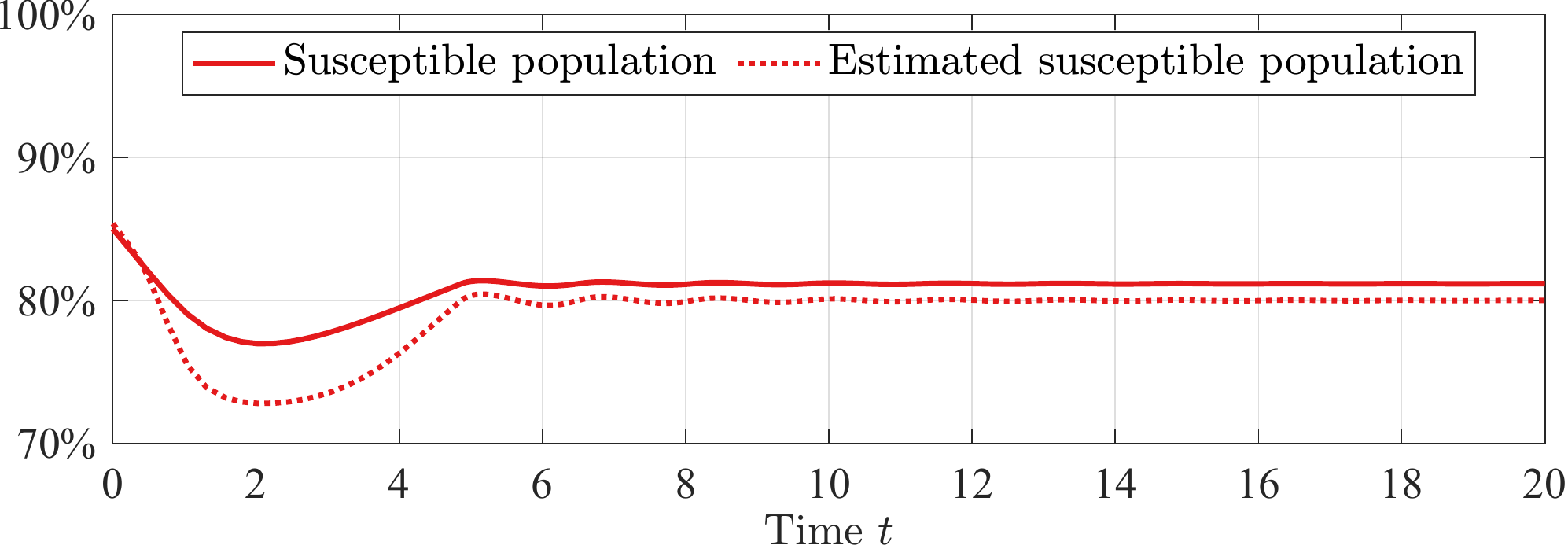}}\\
\centering
\subfigure[]{\includegraphics[width=.95\columnwidth]{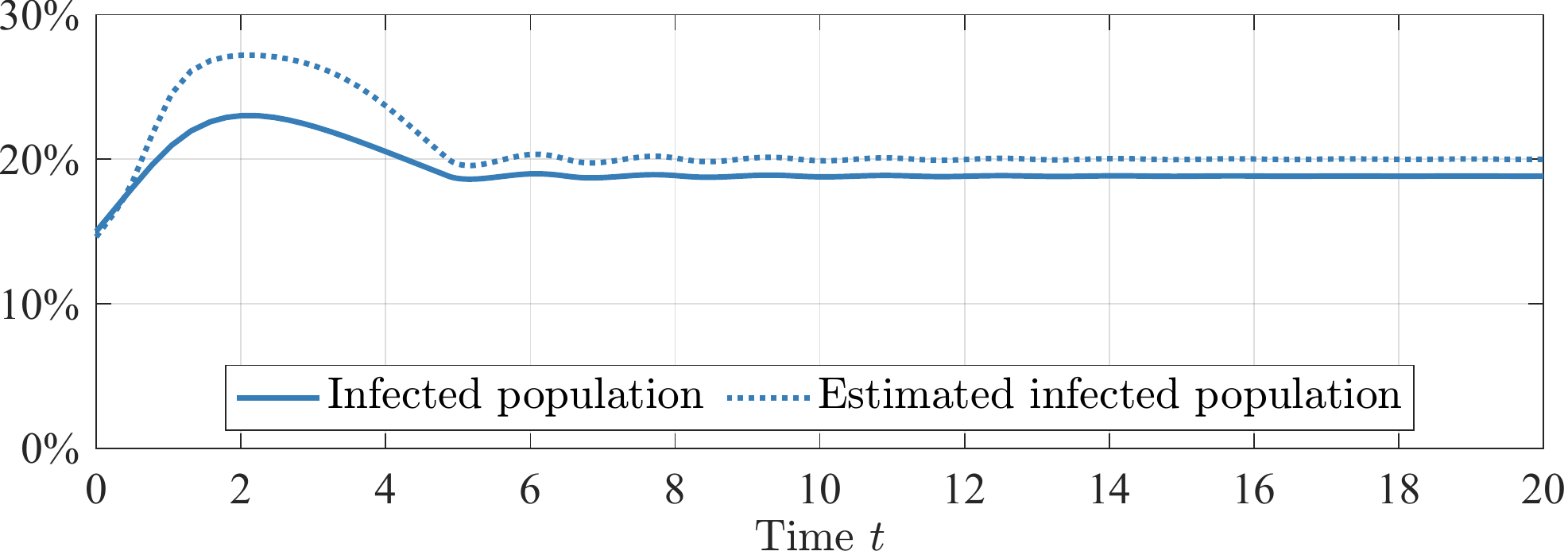}}\\
\centering
\subfigure[]{\includegraphics[width=.95\columnwidth]{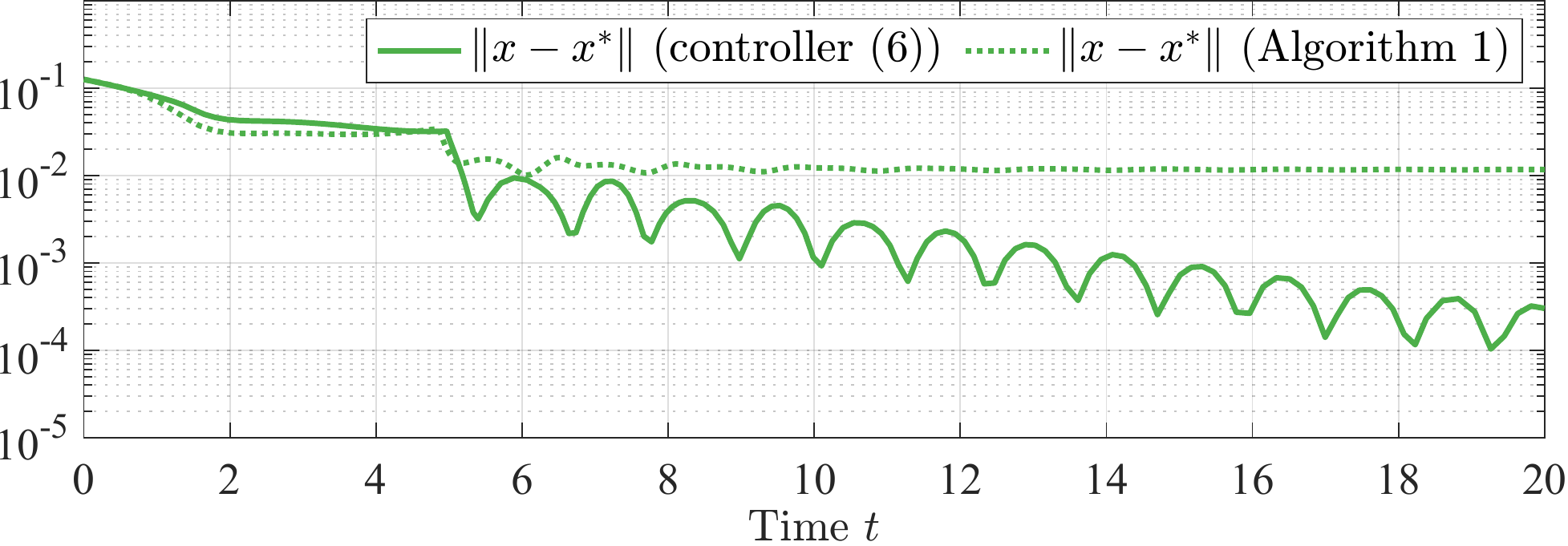}} 
\caption{Numerical validation of Algorithm~\ref{alg:perception_cost}. The 
technique is used to control an epidemic SIS model whose state cannot be 
directly measured and is instead reconstructed using a neural network. Panels 
(a)-(b) compare the time-evolution of the true states and of the estimated 
states. Panel (c) compares the performance of the ideal 
controller~\eqref{eq:ideal_controller} with that of the perception-based 
counterpart described in Algorithm~\ref{alg:perception_cost}. Steady-state 
errors in Algorithm~\ref{alg:perception_cost} are associated to errors 
originated by the use of a neural network (see 
Proposition~\ref{prop:General_perceptionstate_forward}).}
\label{fig:SISstatePerception}
\vspace{-.5cm}
\end{figure}

\subsection{Optimization with Cost-Function Perception}
Next, we illustrate the case of cost perception 
(Section~\ref{sec:costPerception}).
For illustration purposes, we focus on cases where the analytic expression of 
$\phi(u)$ in~\eqref{opt:objectiveproblem} is known, while $\psi(x)$ is unknown.
As described in Algorithm~\ref{alg:cost_perception}, we utilized a set of 
samples $\{(x_i, \psi(x_i))\}_{i=1}^{M}$ to train a feedforward neural network 
to approximate the function $\psi(x)$.
Simulation results are illustrated in  Figure~\ref{fig:SISsimulation}.
Figure~\ref{fig:SISsimulation}(a) illustrates the set of samples used for training
and provides a comparison between the true gradient $\nabla \psi(x)$ and the 
approximate gradient $g_x(x)$, obtained through the neural network.
Figure~\ref{fig:SISsimulation}(b) provides a comparison between the state 
trajectories obtained by using the ideal controller~\eqref{eq:ideal_controller} 
and those obtained through the perception-based controller in 
Algorithm~\ref{alg:cost_perception}.
Figure~\ref{fig:SISstatePerception}(c) provides a comparison 
of the tracking error between the use of the ideal 
controller~\eqref{eq:ideal_controller} (which uses the exact gradients of the 
cost) and the controller in Algorithm~\ref{alg:cost_perception}.
The simulations illustrate that while the exact 
controller~\eqref{eq:ideal_controller} is capable of converging 
to the desired optimizer with arbitrary accuracy, the controller in 
Algorithm~\ref{alg:cost_perception} yields a nontrivial steady-state error 
due to uncertainties in the gradient function. Overall, the simulations
validate the convergence claim made in 
Proposition~\ref{prop:General_perceptioncost_forward}.

\begin{figure}[t]
\subfigure[]{\includegraphics[width=.47\columnwidth]{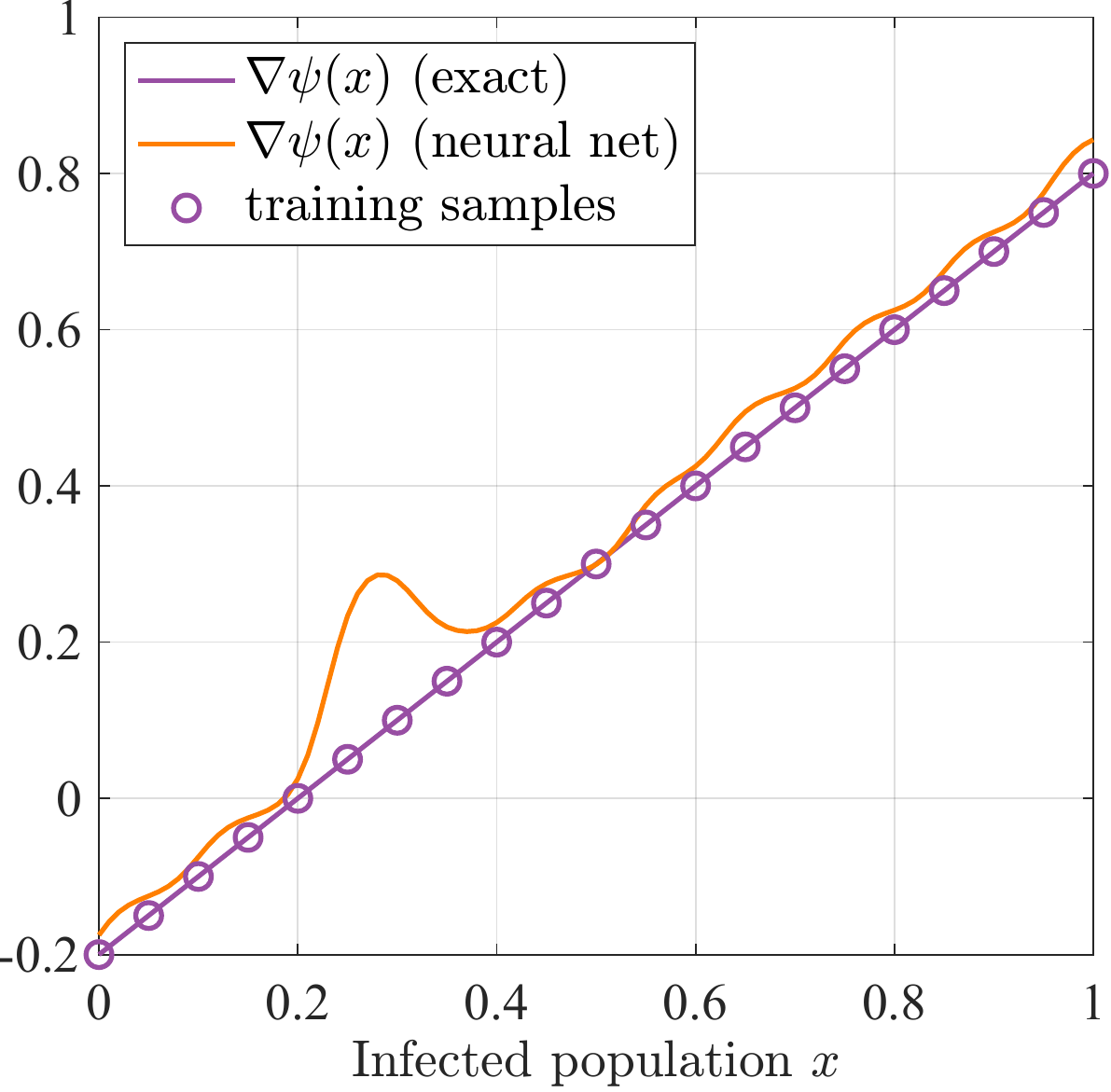}}  \hfill
\subfigure[]{\includegraphics[width=.52\columnwidth]{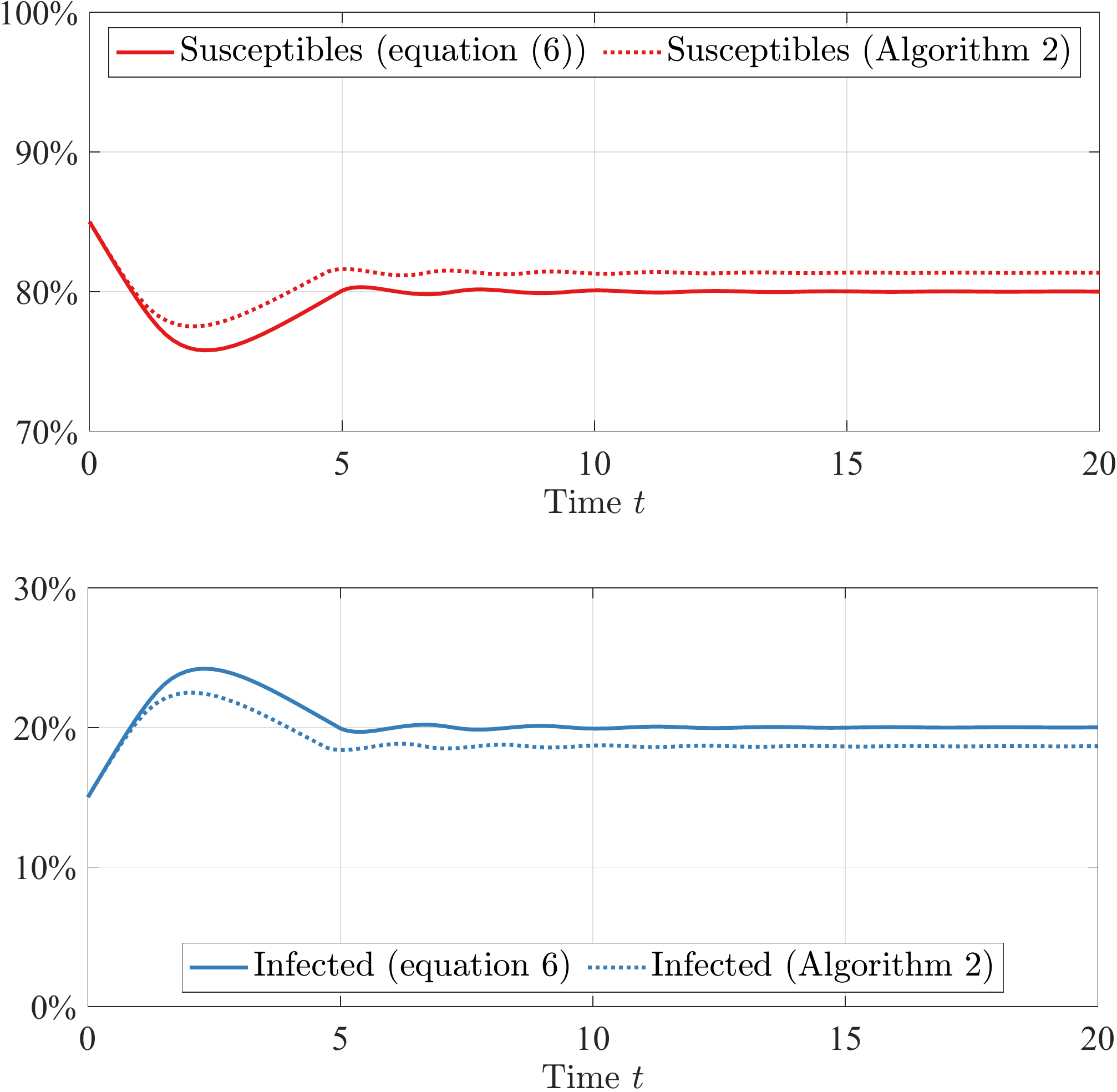}}\\
\subfigure[]{\includegraphics[width=\columnwidth]{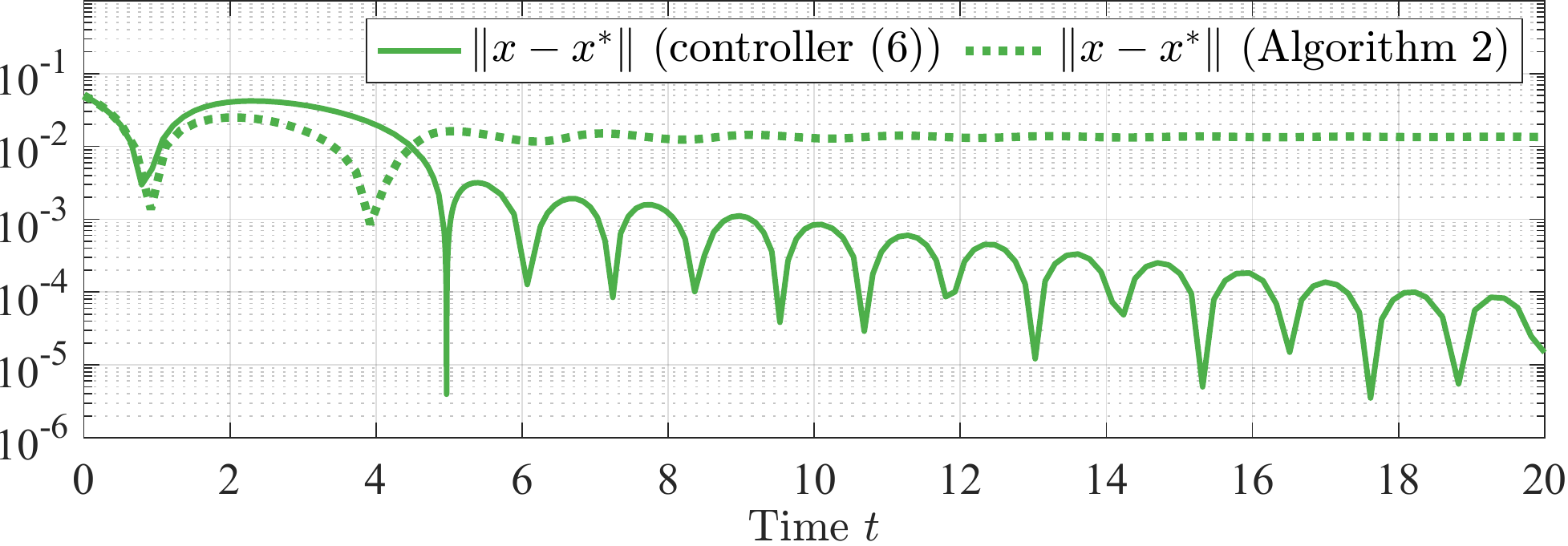}} 
\caption{Numerical validation of Algorithm~\ref{alg:cost_perception}. The 
technique is used to control an epidemic SIS model, where the economic impact of
 infections is unknown and must be estimated from samples using a neural 
network. Panel (a) illustrates the approximation accuracy of the neural network 
(used to approximate the gradient of the function). Panel~(b) compares the 
time-evolution of the true states obtained by using 
Algorithm~\ref{alg:cost_perception} and those obtained by using the ideal 
controller~\eqref{eq:ideal_controller}. Panel (c) compares the tracking error of
the ideal and perception-based control methods.
Steady-state errors in Algorithm~\ref{alg:cost_perception} are associated to 
errors originated by the use of a neural network to compute the gradients  (see 
Proposition~\ref{prop:General_perceptioncost_forward}).}
\label{fig:SISsimulation}
\vspace{-.5cm}
\end{figure}

\section{Conclusions}
\label{sec:conclusions}

We proposed algorithms to control and optimize dynamical systems when the 
system state cannot be measured and the cost functions in the optimization are 
unknown, instead, these quantities can only be sensed via neural network-based 
perception.
Our results show for the first time how feedback-based optimizing controllers 
can be adapted to operate with perception in the control loop.
Our findings crucially hinge on recent results on the uniform approximation 
properties of deep neural networks.
We believe that the results can be further refined to account for cases where 
the training of the neural networks is performed online and we are currently 
investigating such possibility. While this paper provided conditions to obtain exponential ISS results for the interconnection of a plant with a projected gradient-flow controller, future efforts will analyze the interconnection of plants and controllers that are (locally) asymptotically stable when taken individually. 

\appendix[Proof of Theorem~\ref{thm:General_stability}]

The proof of this claim uses Lyapunov-based singular perturbation reasonings, 
and is organized into seven main~steps.

\noindent
\textit{(1 -- Change of variables)}
We consider the change of variables \( \tilde x = x - h(u,w)\), which shift the 
equilibrium of~\eqref{eq:plantModel} to the origin. In the new 
variables,~\eqref{eq:closedloop_error}~reads as:
\begin{subequations}
\label{eq:shiftedPlantAndController}
\begin{align}
\label{eq:shiftedPlantAndController-a}
\dot{\tilde x} & =  f(\tilde x + h(u,w),u,w) 
- \frac{d}{dt} {h}(u,w),\\
\label{eq:shiftedPlantAndController-b}
\dot u & = \proj \{ u - \eta( \tilde F(\tilde x, u)  
+ e(\tilde x + h(u,w_t),u) \} -  u,  
\end{align}
\end{subequations}
where we denoted, in compact form, 
$\tilde F(\tilde x, u) := F(\tilde x + h(u,w), u)$,
$\tilde e(\tilde x,u) := e(\tilde x + h(u,w_t),u)$.
Before proceeding, we note that:
\begin{align}\label{eq:boundF}
\|\tilde F(\tilde x,u) - \tilde F(0,u_t^*)\| &\leq \nonumber
\|\tilde F(\tilde x,u) - \tilde F(0,u)\| + \\
&\quad\quad\quad\quad\quad \|\tilde F(0,u) - \tilde F(0,u_t^*)\|\nonumber\\
&\leq \ell_{h_u} \ell_y \|\tilde x\| + \ell \|u - u_t^*\|, 
\end{align}
where we recall that $(\ell_{h_u}, \ell_y, \ell)$
are as in Assump.~\ref{as:steadyStateMap} and~\ref{as:LipsConvex}(a).

\noindent
\textit{(2 -- Lyapunov functions)}
Inspired by singular perturbation reasonings~\cite[Ch. 11]{Khalil:1173048}, we 
adopt the following composite Lyapunov function 
for~\eqref{eq:shiftedPlantAndController}:
\begin{align} 
\nu(\tilde x, u, t) 
:= \theta \frac{1}{\eta} W(\tilde x, u, w_t) +
(1 - \theta) \frac{1}{\eta} V(u, u_t^*), 
\end{align}
where $\theta$ is as defined in~\eqref{eq:uglyConstants}, 
$W(\tilde x, u, w)$ describes a Lyapunov function for~\eqref{eq:shiftedPlantAndController-a} and is 
given in Lemma~\ref{lem:converse_plant}, and  $V(u, u_t^*)$ describes a Lyapunov
function for~\eqref{eq:shiftedPlantAndController-b} and is given by:
$ V(u,u_t^*) = \frac{1}{2}\| u - u_t^*\|^2$.

Before proceeding, we notice that $\nu(\tilde x, u, t) $ satisfies
\begin{align*}
c_1 \|(\tilde x,u-u_t^*)\|^2 \leq \nu(\tilde x, u, t)  \leq c_2\|(\tilde x,u-u_t^*)\|^2,
\end{align*}
where $c_1,c_2$ are as in~\eqref{eq:uglyConstants}.

\noindent
\textit{(3 -- Bound for the Time-Derivative of $W$)}
We begin by bounding the time-derivative $\frac{d}{dt} W$.
To this end, notice that Lemma~\ref{lem:converse_plant} guarantees 
the following estimates:
\begin{align*}
&\frac{\partial W}{\partial \tilde x}
f(\tilde x + h(u,w),u,w) 
\leq -d_3 \norm{\tilde x}^2,\\
&\frac{\partial W}{\partial \tilde x} \frac{d}{dt} {h}(u,w)
= \frac{\partial W}{\partial \tilde x}
\left( \nabla_u h(u,w) \dot u + \nabla_w h(u,w) \dot w \right)\\
&\quad\quad\quad\quad\quad\quad
\leq d_4  \norm{\tilde x} (\ell_{h_u} \norm{\dot u}  + \ell_{h_w} \norm{\dot w}).
\end{align*}
By using the above estimates and Lemma~\ref{lem:converse_plant}, we have:
\begin{align*}
\frac{d}{dt} W (\tilde x, u, w) &\leq 
-d_3 \|\tilde x\|^2 
+ (d_4 \ell_{h_u} +d_5) \|\dot u\| \|\tilde x\| \nonumber\\
&\quad\quad
+ (d_4 \ell_{h_w}+ d_6) \|\dot w_t\| \|\tilde x\|.
\end{align*}
Next, refine the bound on the second term. Notice that
\begin{align*}
\| \dot u \| &= \| \proj \left\{ u - \eta \tilde F(\tilde x,u) + \eta \tilde e(\tilde x, u) \right\} - u \|\\
&\leq \eta \| \tilde F(\tilde x, u) - \tilde F(0,u_t^*)\| + \eta \|\tilde e(\tilde x, u)\|\\
&= \eta \ell_{h_u}  \ell_y \|\tilde x\| + \eta \ell \|u - u_t^*\| + \eta \|\tilde e(\tilde x, u)\|.
\end{align*}
To obtain the first inequality, add $u^* - \proj \{u^* - \eta \tilde F(0,u^*) \} = 0$ and apply the triangle inequality. Then, apply the bounds obtained in \eqref{eq:boundF}.
It follows that
\begin{align}
\label{eq:dotWfinal}
& \frac{d}{dt} W(\tilde x,u,w) \leq 
-d_3 \|\tilde x\|^2 
+ \eta \ell_{h_u}\ell_y( d_4 \ell_{h_u} + d_5) \|\tilde x\|^2 \nonumber\\
&\quad
+ \eta \ell (d_4 \ell_{h_u} +d_5) \|u - u^*\| \|\tilde x\|\\
&\quad
+ \eta (d_4 \ell_{h_u} +d_5) \|e(\tilde x, u)\| \norm{\tilde x}
+ (d_4 \ell_{h_w} +d_6) \| \dot w_t\| \|x\|.\nonumber
\end{align}

\noindent
\textit{(4 -- Bound for the Time-Derivative of $V$)}
Next, we bound the time-derivative $\frac{d}{dt} V$.
In what follows, we use the compact notation $\tilde u := u -u_t^*$.
By expanding:
\begin{align}\label{eq:dotV}
\begin{split}
\frac{d}{dt} V(u,u_t^*) &= 
\tilde u^\top \dot u - \tilde u^\top \dot{u}_t^*
\leq \tilde u^\top \dot u + \ell_J \|\dot w\|\|\tilde u\|,
\end{split}
\end{align}
where we used Assumption~\ref{as:mapu}.
The first term above satisfies:
\begin{align}
\label{eq:auxBoundV}
\tilde u^\top \dot u &\leq 
\tilde u^\tsp (\proj \{u - \eta \tilde F(0,u)\} - u)  \\
& \quad\quad\quad
+ \eta \ell_{h_u}\ell \|\tilde x\|  \|\tilde u\| 
+ \eta \|\tilde u\| \|\tilde e(\tilde x, u)\|.\nonumber
\end{align}
To obtain~\eqref{eq:auxBoundV}, we added and subtracted  
$\proj \{u - \eta F(0,u) \}$, we applied the triangle inequality, and 
we used Assumptions~\ref{as:steadyStateMap} 
and~\ref{as:LipsConvex}(a).
Next, we use the fact that the first term on the right hand side 
of~\eqref{eq:auxBoundV} satisfies:
\begin{align}
\label{eq:auxBoundV2}
\tilde u^\top \left(P_u - u \right) \leq 
 - \eta \left( \mu - \eta \ell^2/4 \right)  \|\tilde u\|^2.
\end{align}
(This fact will be proven in Step 5, shortly below.)
By combining \eqref{eq:dotV}-\eqref{eq:auxBoundV}-\eqref{eq:auxBoundV2} we 
conclude that:
\begin{align}
\label{eq:dotVFinal}
\frac{d}{dt} V &\leq - \eta \left( \mu - \eta \frac{\ell^2}{4} \right) 
\|\tilde u\|^2 
+ \eta \ell_{h_u}\ell \|\tilde x\| \|\tilde u\| \nonumber\\
& \quad\quad + \eta \|\tilde u\| \|e(\tilde x, u)\| + \ell_J \|\dot w\| \|\tilde u\|.
\end{align}

\noindent
\textit{(5 -- Proof of \eqref{eq:auxBoundV2})} For brevity, in what follows we 
use the compact notation $P_u := \proj \{u - \eta \tilde F(0,u) \}.$
To prove~\eqref{eq:auxBoundV2}, we recall the following property of the 
projection operator:
\[ \left( u' - \proj (v) \right)^\top \left( \proj (v) - v \right) \geq 0, ~~\forall~ u' \in \mathcal{C}, \]
which holds for any $v \in \mc \real^{n_u}$.
By applying this property with $u' := u_t^*$ and $v = u - \eta \tilde F(0,u)$,
we have:
\[\left( u_t^* - P_u \right)^\top ( P_u - u + \eta \tilde F(0,u) ) 
\geq 0. \]
By expanding and by adding \( u^\top (u - P_u + \eta \tilde F(0,u) ) \) to 
both sides of the inequality, the left hand side reads as:
\[ u^\top u - u^\top P_u - u_t^{*\top } u + u_t^{*\top} P_u + \eta u^\top \tilde F(0,u) - \eta \tilde F(0,u)^\top P_u,  \]
while the right hand side reads as:
\[u^\top u - P_u^\top u - u^\top P_u^\top +  P_u^\top P_u + \eta u^\top \tilde F(0,u) - \eta u_t^{*^\top} \tilde F(u,0), \]
and, after regrouping, 
\begin{align*}
\tilde u^\top \left(u - P_u \right) \geq \|u - P_u\|^2 
+ \eta (P_u - u_t^*)^\top \tilde F(u,0).
\end{align*}
By adding and subtracting 
\( \eta \left( P_u - u_t^*\right)^\top \tilde F(0,u_t^*) \):
\begin{align*}
\tilde u^\top (P_u - u ) &\leq 
- \|u - P_u\|^2 - \eta (P_u - u_t^*)^\top \tilde F(u_t^*,0)\\
& \quad
- \eta (P_u - u_t^*)^\top ( \tilde F(0,u) - \tilde F(0,u_t^*) ).
\end{align*}
From this, expand \( - (P_u - u_t^*) = u - P_u - \tilde u \) and apply 
\( \left( P_u - u_t^* \right)^\top \tilde F(0,u) \geq 0. \) Use strong convexity and Lipschitz-continuity to obtain, 
\begin{align*}
\tilde u^\top \left(P_u - u \right) &\leq 
- \|u - P_u\|^2 - \eta \mu \|\tilde u\|^2 
+ \eta \ell \|u - P_u\| \|\tilde u\|\\
&\leq 
 - \eta \left( \mu - \eta \ell^2/4 \right)  \|\tilde u\|^2,
\end{align*}
where the second inequality follows by recalling that, for real numbers $a,b$, 
it holds \( 2ab < a^2 + b^2 \), and by letting $a = \|u - P_u\|$ and 
$b = \frac{1}{2}\eta \ell \|\tilde u\|$.
Hence, \eqref{eq:auxBoundV2} is proved.

\noindent
\textit{(6 -- Conditions for Exponential Convergence)}
By combining \eqref{eq:dotWfinal} and \eqref{eq:dotVFinal}:
\begin{align}\label{eq:firstBoundNu}
\frac{d}{dt} \nu &\leq 
(1 - \theta) \{ -\left(\frac{\mu - \eta \ell^2}{4} \right) \|\tilde u\|^2  
+ \|\tilde e(\tilde x,u)\| \| \tilde u\| \\
&~~+ \frac{\ell_J}{\eta}\|\dot w_t\| \|\tilde u\|
+ \ell_{h_u}\ell \|\tilde x\| \|\tilde u\|\}\nonumber\\
& ~~  + \theta \{ - \frac{d_3}{\eta} \|\tilde x\|^2 
+\ell_{h_u}\ell_y( d_4 \ell_{h_u} + d_5) \|\tilde x\|^2 \nonumber\\
&\quad
+ \ell (d_4 \ell_{h_u} +d_5) \|\tilde u\| \|\tilde x\|\nonumber\\
&\quad
+ (d_4 \ell_{h_u} +d_5) \|\tilde e(\tilde x, u')\|
+ \frac{(d_4 \ell_{h_w} +d_6)}{\eta} \| \dot w_t\| \|x'\|\}.\nonumber
\end{align}
By using \eqref{eq:ControllerGain}, we have $\eta < \frac{2\mu}{\ell^2}$, and thus
the first term satisfies 
$(\frac{\eta \ell^2}{4} - \mu ) \|\tilde u\|^2 \leq - \frac{\mu}{2} \norm{\tilde u}^2$. Next, let $s \in (0,1)$ be a fixed constant.  Then,~\eqref{eq:firstBoundNu}
can be rewritten as:
\begin{small}
\begin{align}\label{eq:firstBoundNu2}
\frac{d}{dt} \nu &\leq 
(1 - \theta) \{ - \frac{(1-s)\mu}{2} \|\tilde u\|^2  
+ \|\tilde e(\tilde x,u)\| \| \tilde u\| \nonumber\\
&~~+ \frac{\ell_J}{\eta}\|\dot w_t\| \|\tilde u\|
+ \ell_{h_u}\ell \|\tilde x\| \|\tilde u\|\}\nonumber\\
& ~~  + \theta \{ - \frac{(1-s)d_3}{\eta} \|\tilde x\|^2 
+\ell_{h_u} \ell_y (d_4  \ell_{h_u} +d_5) \|\tilde x\|^2 \nonumber\\
&~~+ \ell(d_4 \ell_{h_u} +d_5) \|\tilde u\| \|\tilde x\| \nonumber\\
& ~~ + (d_4 \ell_{h_u} +d_5) \|\tilde e(\tilde x,u) \|\tilde x\| 
+ \frac{d_4\ell_{h_w} + d_6}{\eta}  \| \dot w_t\| \|\tilde x\| \}\nonumber\\
&~~- \underbrace{\min \{ s \mu \eta, s \frac{d_3}{d_2} \}}_{=c_0}\nu,
\end{align}
\end{small}
where we used $-(1 - \theta) s\frac{\mu}{2}\|\tilde u\|^2 = 
-(1 - \theta) s \mu V(u,u_t^*)$ and 
$- \theta s \frac{d_3}{\eta} \|x'\|^2 \leq - \frac{\theta s}{\eta} \frac{d_3}{d_2} \|x'\|^2 \leq - \frac{\theta s}{\eta} \frac{d_3}{d_2} W$.
The bound~\eqref{eq:firstBoundNu2} can be rewritten as:
\begin{small}
\begin{align}\label{eq:firstBoundNu3}
\frac{d}{dt} \nu &\leq - c_0 \nu 
- \xi^\top \Lambda \xi 
+ (1 - \theta) \{ \|\tilde e(\tilde x, u)\| \|\tilde u\| 
+ \frac{\ell_J}{\eta} \|\dot w_t\| \|\tilde u\|\}\nonumber\\
&~~+ \theta \{(d_4 \ell_{h_u} + d_5))\|\tilde e(\tilde x, u)\| \|\tilde x\| 
+ \frac{d_4 \ell_{h_w}+d_6}{\eta} \| \dot w_t\| \|\tilde x\|  \},
\end{align}
\end{small}
where $\xi = (\norm{\tilde u}, \norm{\tilde x})$ and
\begin{align*}
    \Lambda &=\left(\begin{array}{ll}(1 - \theta)\alpha_1 & -\frac{1}{2}[(1-\theta)\beta_1 + \theta \beta_2] \\ -\frac{1}{2}[(1-\theta)\beta_1 + \theta \beta_2] & \theta[\frac{\alpha_2}{\eta} - \beta_1\beta_2]
    \end{array}\right),
\end{align*}
with $\alpha_1 := (1 - s)\frac{\mu}{2}$, 
$\alpha_2 := (1 - s)d_3$, 
$\beta_1 := \ell \ell_{h_u}$, 
$\beta_2 := \ell(d_4 \ell_{h_u}+d_5)$, and
$\gamma := \ell_y \ell_{h_u}(d_4 \ell_{h_u} + d_5)$.
Finally, we will show that $\xi^\top \Lambda \xi\leq 0$. Equivalently,
$\Lambda$ is positive definite if and only if its principal minors are positive, 
namely,
\begin{align*}
    [(1-\theta)\alpha_1]&[\theta \left(\frac{\alpha_2}{\eta}- \beta_1  \beta_2\right)] > \frac{1}{4}[(1-\theta)\beta_1 + \theta \beta_2]^2,
\end{align*}
or, equivalently,
\begin{equation*}
    \eta < \frac{ \alpha_1\alpha_2}{\frac{1}{4\theta (1-\theta)}[(1-\theta)\beta_1 + \theta \beta_2]^2 + \alpha_1 \beta_1  \beta_2}. 
\end{equation*}
The right hand side is a concave function of $\theta$ and, by maximizing it with
respect to $\theta$, we have that the maximum is obtained at 
$\theta^* = \frac{\beta_1}{\beta_1 + \beta_2},$ which gives
\begin{align*}
\eta < \frac{\alpha_1 \alpha_2}{\alpha_1 \gamma + \beta_1 \beta_2},
\end{align*}
which holds under~\eqref{eq:ControllerGain}.

\noindent
\textit{(7 -- Derivation of Convergence Bound)}
We begin by showing that~\eqref{eq:firstBoundNu3} can be further refined as follows:
\begin{align}\label{eq:firstBoundNu4}
\frac{d}{dt} \nu &\leq - c_0 \nu + c_3 \gamma \nu 
+ \left( c_4 \delta + c_5 \|\dot w_t\| \right) \sqrt{\nu}.
\end{align}
To this end, we first notice that the following inequalities are true: 
$\|\tilde u\| \leq \sqrt{2}\sqrt{V} \leq \sqrt{2} \frac{\sqrt{\eta}}{\sqrt{(1 - \theta)}}\sqrt{\nu}$, 
$\|\tilde x\| \leq \frac{\sqrt{W}}{\sqrt{d_1}} \leq \frac{\sqrt{\eta}}{\sqrt{\theta} \sqrt{d_1}}\sqrt{\nu}$ and that 
$\sqrt{\theta}, \sqrt{1 - \theta} \in (0,1)$ since 
$\theta \in (0,1)$.
Using these facts, the two positive terms in~\eqref{eq:firstBoundNu3} satisfy:
\begin{align*}
\frac{(1 - \theta)}{\eta}\ell_J \|\dot w_t\| \|\tilde u\| &
\leq \eta^{1/2} \sqrt{2}\ell_J \|\dot w_t\| \sqrt{\nu}\\
\theta \frac{d_4 \|\nabla_w h(u,w)\|}{\eta} \|\dot w_t\| \|\tilde x\| &
\leq d_4 \ell_{h_w} \frac{\| \dot w_t\|}{\sqrt{\eta}\sqrt{d_1}} \sqrt{\nu},
\end{align*}
from which \eqref{eq:firstBoundNu4} follows with $c_3, c_4, c_5$ as 
in~\eqref{eq:uglyConstants}.

To conclude, define  $\nu' := \sqrt{\nu}$. Then,
\begin{align*}
\dot \nu' &\leq -\frac{1}{2}\left(c_0  - c_3 \gamma \right) \nu' + \frac{1}{2}\left( c_4 \delta + c_5 \|\dot w_t\| \right).
\end{align*}
Define \( \Phi(t,t_0) := \exp{ \left(-\frac{1}{2} (c_0-c_3) (t - t_0)\right) }\)
for $t \geq t_0.$ By the comparison lemma~\cite{Khalil:1173048}: 
\begin{align*}
\dot \nu' &\leq \Phi(t,t_0) \dot \nu'(t_0) + \int_{t_0}^t \Phi(t,\tau) \left(\frac{c_4}{2} \delta + \frac{c_5}{2}\|\dot w_{\tau}\| \right) d \tau,
\end{align*}
which implies, by recalling \(c_1 \|z\|^2 \leq \nu \leq c_2 \|z\|^2 \),  that
\begin{align*}
\|z(t)\| &\leq 
\sqrt{\frac{c_2}{c_1}} \Phi(t,t_0) \|z(t_0)\| 
+ \frac{c_4}{2 \sqrt{c_1}} \int_{t_0}^t \Phi(t,\tau) \delta d \tau \\
&\quad\quad + \frac{c_5}{2 \sqrt{c_1}}\int_{t_0}^t \Phi(t,\tau) \| \dot w_{\tau}\| d \tau,
\end{align*}
from which~\eqref{eq:genErrorBound} follows.


\bibliographystyle{IEEEtran}
\bibliography{alias, bibliography,main_GB,biblio_perception.bib}

\end{document}